\documentclass[12pt]{article}

\usepackage{amsmath}
\usepackage{amsthm}
\usepackage{amssymb}
\usepackage{cite}
\usepackage{url}
\usepackage[multiple]{footmisc}
\usepackage{graphicx}
\usepackage{epstopdf}
\usepackage{chngcntr}
\usepackage{enumitem}
\usepackage{hhline}
\usepackage{array}
\usepackage{hyperref}
\usepackage{color}
\usepackage{array}
\usepackage{multirow}

\hypersetup{colorlinks=false}

\if@twoside
\else 
\fi

\numberwithin{equation}{section}
\counterwithin{figure}{section}
\counterwithin{table}{section}

\theoremstyle{plain}
\newtheorem{theorem}{Theorem}[section]
\newtheorem{lemma}[theorem]{Lemma}
\newtheorem{proposition}[theorem]{Proposition}

\newtheorem{corollary}[theorem]{Corollary}

\theoremstyle{definition}
\newtheorem{definition}{Definition}[section]
\newtheorem{assumption}{Assumption}[section]

\theoremstyle{remark}
\newtheorem{remark}{Remark}[section]

\newcommand{\norm}[1]{\left\|#1\right\|}
\newcommand{\abs}[1]{\left\vert#1\right\vert}
\newcommand{\spr}[1]{\left\langle\,#1\,\right\rangle}
\newcommand{\kl}[1]{\left(#1\right)}
\newcommand{\Kl}[1]{\left\{#1\right\}}

\newcommand{\supp}[1]{\text{supp}\left(#1\right)}

\definecolor{aog}{rgb}{0.0, 0.5, 0.0}

\newcommand{\R}{\mathbb{R}} 
\newcommand{\C}{\mathbb{C}}
\newcommand{\N}{\mathbb{N}}
\newcommand{\Z}{\mathbb{Z}}

\newcommand{\xd}{x^\dagger}
\newcommand{\xa}{x_\alpha}
\newcommand{\xad}{x_\alpha^\delta}

\newcommand{\yd}{y^{\delta}}

\newcommand{\OA}{{\Omega_A}}
\newcommand{\Ol}{{\Omega_l}}
\newcommand{\OT}{{\Omega_T}}

\newcommand{\OAaghl}{{\OA(\aghl)}}

\newcommand{\Lt}{{L_2}}

\newcommand{\LtOl}{{L_2(\Ol)}}
\newcommand{\LtOT}{{L_2(\OT)}}

\newcommand{\LtOA}{{L_2(\OA)}}

\newcommand{\ag}{\alpha_g}
\newcommand{\agx}{\alpha_g^x}
\newcommand{\agy}{\alpha_g^y}
\newcommand{\aghl}{\ag h_l}

\newcommand{\vphi}{\varphi}

\newcommand{\clg}{c_{l,g}}

\newcommand{\At}{\tilde{A}}

\newcommand{\et}{\tilde{e}}
\newcommand{\lt}{{\ell_2}}

\newcommand{\AD}{\mathcal{A}}
\newcommand{\ADb}{\bar{\mathcal{A}}}
\newcommand{\Ad}{A^\dagger}

\newcommand{\D}{\mathcal{D}}

\newcommand{\Ft}{\tilde{F}}

\newcommand{\fkt}{\tilde{f}_k}
\newcommand{\ekt}{\tilde{e}_k}
\newcommand{\sk}{\sigma_k}

\newcommand{\mukj}{\mu_{k,j}}
\newcommand{\rk}{{r_k}}
\newcommand{\vkl}{\boldsymbol{v}_{k,l}}
\newcommand{\ukl}{\boldsymbol{u}_{k,l}}
\newcommand{\vkj}{\boldsymbol{v}_{k,j}}
\newcommand{\ukj}{\boldsymbol{u}_{k,j}}
\newcommand{\ak}{\alpha_k}

\newcommand{\xb}{\bar{x}}

\newcommand{\ol}[1]{\overline{#1}}

\newcommand{\OD}{\Omega_D}
\newcommand{\OS}{\Omega_S}

\newcommand{\LtOo}{{L_2(-1,1)}}

\newcommand{\LtR}{{L_2(\R)}}
\newcommand{\LtT}{{L_2(0,2\pi)}}
\newcommand{\LtOD}{{L_2(\Omega_D)}}
\newcommand{\LtOS}{{L_2(\Omega_S)}}

\newcommand{\ejkl}{e_{j,k,l}}
\newcommand{\fjkl}{f_{j,k,l}}
\newcommand{\ejklt}{\tilde{e}_{j,k,l}}
\newcommand{\fjklt}{\tilde{f}_{j,k,l}}

\newcommand{\LtRN}{{L_2(\R^N)}}

\newcommand{\HaRN}{{H^\alpha(\R^N)}}
\newcommand{\HbOS}{{H^\beta(\OS)}}

\newcommand{\HapOS}{{H^{\alpha+1/2}(\OS)}}

\newcommand{\HazO}{{H^\alpha_0(\Omega)}}
\newcommand{\HazOo}{{H^\alpha_0(-1,1)}}

\newcommand{\HazOD}{{H^\alpha_0(\OD)}}
\newcommand{\HaR}{{H^\alpha(\R)}}

\newcommand{\HaOS}{{H^\alpha(\OS)}}

\newcommand{\ltN}{{\lt(\N)}}

\newcommand{\Lk}{\boldsymbol{\Lambda}_k}
\newcommand{\lk}{\lambda_k}
\newcommand{\lkc}{\overline{\lambda}_k}
\newcommand{\lkmn}{\lambda_k^{m,n}}
\newcommand{\Lkb}{\boldsymbol{\bar{\Lambda}}_k}

\newcommand{\mb}{{\bar{m}}}

\newcommand{\mcX}{\multirow{2}{*}{X}}

\title{Frame Decompositions of Bounded Linear Operators in Hilbert Spaces with Applications in Tomography}
\author{
Simon Hubmer\footnote{Johann Radon Institute Linz, Altenbergerstra{\ss}e 69, A-4040 Linz, Austria, (simon.hubmer@ricam.oeaw.ac.at), Corresponding author.} ,
Ronny Ramlau\footnote{Johannes Kepler University Linz, Institute of Industrial Mathematics, Altenbergerstra{\ss}e 69, A-4040 Linz, Austria, (ronny.ramlau@jku.at)} \footnote{Johann Radon Institute Linz, Altenbergerstra{\ss}e 69, A-4040 Linz, Austria, (ronny.ramlau@ricam.oeaw.ac.at)}
}

\begin{document}

\maketitle

\begin{abstract}
We consider the decomposition of bounded linear operators on Hilbert spaces in terms of functions forming frames. Similar to the singular-value decomposition, the resulting frame decompositions encode information on the structure and ill-posedness of the problem and can be used as the basis for the design and implementation of efficient numerical solution methods. In contrast to the singular-value decomposition, the presented frame decompositions can be derived explicitly for a wide class of operators, in particular for those satisfying a certain stability condition. In order to show the usefulness of this approach, we consider different examples from the field of tomography.

\end{abstract}


\section{Introduction}

In this paper, we consider bounded linear operators
	\begin{equation*}
		A \, : \, X \to Y \,,
	\end{equation*}
between real or complex Hilbert spaces $X$ and $Y$. These may be of the general form
	\begin{equation}\label{def_X_Y}
		X := \prod\limits_{m=1}^M X_m \,,
		\qquad \text{and} \qquad
		Y := \prod\limits_{n=1}^N Y_n \,,
	\end{equation}	
where $X_m$ and $Y_n$ are again Hilbert spaces, and $M,N \in \N$. In particular, we are interested in solving potentially ill-posed linear operator equations of the form
	\begin{equation}\label{Ax=y}
		A x = y  \,.
	\end{equation}
An element $x^* \in X$ is commonly called a \emph{least-squares solution} of \eqref{Ax=y} if
	\begin{equation*}
		\norm{A x^* - y} = \inf\Kl{\norm{Ax - y} \, \vert \, x \in X} \,,
	\end{equation*}	
and a \emph{solution} if $Ax^*=y$. Furthermore, the \emph{minimum-norm (least-squares) solution} $\xd$ is defined as the unique (least-squares) solution of minimal norm, i.e.,
	\begin{equation*}
		\norm{\xd} = \inf\Kl{\norm{x^*} \, \vert \, x^* \text{ is a (least-squares) solution of } Ax = y}\,.
	\end{equation*}		
In case that $A$ is compact, it is well-known \cite{Conway_1994,Heuser_1981,Alt_2016,Engl_1997} that there exists a \emph{singular system} $(\sigma_k, v_k, u_k)_{k=1}^\infty$ such that $A$ admits a \emph{singular-value decomposition} (SVD) of the form
	\begin{equation}\label{def_A_dec}
		A x  = \sum_{k=1}^{\infty} \sigma_k \spr{x,v_k} u_k \,.
	\end{equation}
Thereby, the \emph{singular values} $\sigma_n$ and \emph{singular functions} $u_k$, $v_k$ are defined via
	\begin{equation}\label{def_singular_system}
		A^*A\, v_k = \sigma_k^2 \,v_k \,, \qquad u_k := (1/\sigma_k) A v_k \,.
	\end{equation} 
The singular system contains all relevant information about the operator $A$ and can be used to characterize the minimum-norm (least-squares) solution $\xd$ of \eqref{Ax=y} via
	\begin{equation}\label{def_Ad}
		\xd := A^\dagger y:= \sum_{k=1}^{\infty} \frac{\spr{y,u_k}}{\sigma_k} v_k \,.
	\end{equation}
Since $\{u_k\}_{k\in\N}$ and $\{v_k \}_{k \in \N}$ are complete orthonormal systems spanning $\ol{R(A)}$ and $N(A)^\perp$, respectively, it follows that $\xd$ is well-defined if and only if the so-called \emph{Picard condition} holds:
	\begin{equation}\label{Picard}
		\sum_{k=1}^{\infty} \frac{\abs{\spr{y,u_k}}^2}{\sigma_k^2} < \infty \,.
	\end{equation}
The rate of decay of the singular values $\sigma_k$ thus defines the degree of ill-posedness of the problem. Furthermore, in the presence of noisy data $\yd$, one can determine a stable approximation $\xad$ of $\xd$ for example via (see e.g.~\cite{Engl_Hanke_Neubauer_1996,Louis_1989})
	\begin{equation}\label{def_xad}
		\xad := \sum_{k=1}^{\infty} g_\alpha(\sigma_k) \spr{\yd,u_k}v_k \,,
	\end{equation}
where $g_\alpha$ is a properly selected approximation of $s \mapsto 1/s$ and the \emph{regularization parameter} $\alpha$ is suitably chosen in dependence on the \emph{noise level} $\delta$, which is such that
	\begin{equation*}
		\norm{y - \yd} \leq \delta \,.
	\end{equation*}
Different choices of $g_\alpha$ give rise to different regularization methods, for example \emph{Tikhonov regularization}, \emph{Landweber iteration}, or the \emph{truncated singular-value expansion}\cite{Engl_Hanke_Neubauer_1996,Louis_1989}.

Hence, if the singular system of $A$ is known or can be easily derived, it makes sense to use it for both analytical considerations as well as numerical implementations. Unfortunately, in many cases an explicit form of the singular system is unknown or hard to derive. The reasons for this are two-fold. On the one hand, finding an explicit representation of the solutions of the eigenvalue equation in \eqref{def_singular_system} is often very difficult. And even if found, the numerical computation might be impossible (see e.g.~\cite{Ramlau_Koutschan_Hofmann_2020}). On the other hand, the domains over which the Hilbert spaces $X$ and $Y$ are defined are often not regular, which makes it difficult to find orthonormal bases which can serve as candidates or building blocks for the singular functions $u_k$ and $v_k$. Thus, while in theory the singular-value decomposition is a great tool for the analysis and solution of inverse problems, it is often not used in practice, save perhaps for finite-dimensional problems, for which the numerical computation/approximation of the then singular vectors quickly becomes overwhelming when the problem is medium- to large-scale.

Hence, in this paper, we generalize and extend the applicability of the singular-value decomposition by considering \emph{frame decompositions} of the operator $A$. By this, we mean that we decompose the operator in a similar way as in \eqref{def_A_dec}, but now with the sets of functions $\{v_k\}_{k\in\N}$ and $\{u_k\}_{k\in\N}$ no longer being orthonormal systems, but only forming (suitably connected) frames over $X$ and $Y$, respectively. Due to the properties of frames we are then able to characterize the minimum-coefficient (least-squares) solution of \eqref{Ax=y} in a way similar to \eqref{def_Ad}. The advantage of this approach is the greater freedom gained by the use of frames over orthonormal systems, with which it is also easier to work over irregular domains. This freedom, combined with the availability of a wide number of highly specialized frames, allows us to obtain frame decompositions of operators for which singular-value decompositions are unavailable so far. These include in particular those operators which satisfy a certain stability property (see \eqref{cond_helper_1} below), as well as continuously invertible linear operators. Furthermore, in order to show that frame decompositions are not only a theoretical possibility, we also present a number of explicit examples from the field of tomography.

Note that frames have been used for the solution of inverse problems before, most commonly in the form of bi-orthogonal or orthonormal wavelet frames (see for example \cite{Chaux_Combettes_Pesquet_Wajs_2007,Teschke_2005,Teschke_2007,Candes_Donoho_2002,Dahmen_1997,Stevenson_2003,Neira_Rubio_Fern_Plastino_1997,Ram_Cohen_Elad_2014, Zhariy_2009,Donoho_1995,Abramovich_Silverman_1998,Kudryavtsev_Shestakov_2019} and the references therein). The popularity of wavelets is due to the fact that they can naturally be used to express sparsity properties, structural aspects, and smoothness assumptions on the sought for solutions. In this connection, we in particular want to mention the Wavelet-Vaguelette Decomposition \cite{Donoho_1995,Abramovich_Silverman_1998,Kudryavtsev_Shestakov_2019}, which is another generalization of the SVD based not on general frames, but specifically on an orthonormal wavelet basis and two bi-ortogonal ``near-orthogonal sets'' (i.e.\ frames).  
	
The outline of this paper is as follows: In Section~\ref{sect_Frames}, we review some necessary material on frames in Hilbert spaces, which we then use in Section~\ref{sect_frame_dec} to derive our main results on frame decompositions. In Section~\ref{sect_discussion}, we show the applicability of the developed theory to specific operator classes, providing a general recipe for their frame decomposition. In Section~\ref{sect_numerical_results}, we apply our results to a number of problems in atmospheric and computerized tomography, before ending with a short conclusion in Section~\ref{sect_conclusion}.

\section{Frames in Hilbert Spaces}\label{sect_Frames}

Before considering frame decompositions of the operator $A$, we first need to recall some basic facts on frames in Hilbert spaces. This short summary, based on the seminal work \cite{Daubechies_1992}, is adapted from our previous publication \cite{Hubmer_Ramlau_2020}. First, recall the following

\begin{definition}\label{def_frame}
A sequence $\{e_k\}_{k \in \N}$ in a Hilbert space $X$ is called a frame over $X$, if and only if there exist numbers $B_1,B_2 > 0$ such that for all $x \in X$ there holds
	\begin{equation}\label{eq_frame}
		B_1 \norm{x}_X^2 \leq \sum\limits_{k=1}^\infty \abs{\spr{x,e_k}_X}^2 \leq B_2 \norm{x}_X^2 \,.
	\end{equation}
The numbers $B_1,B_2$ are called \emph{frame bounds}. The frame is called \emph{tight} if $B_1 = B_2$.
\end{definition}

For a given frame $\{e_k\}_{k\in \N}$, one can consider the so-called \emph{frame (analysis) operator} $F$ and its adjoint \emph{(synthesis)} operator $F^*$, which are given by
	\begin{equation}\label{def_F_Fad}
	\begin{split}
		&F \, : \, X \to \lt(\N) \,, \qquad
		x \mapsto \Kl{\spr{x,e_k}}_{k \in \N} \,,
		\\ 
		&F^* \, : \, \lt(\N) \to X \,, \qquad
		\Kl{a_k}_{k\in\N} \mapsto \sum\limits_{k=1}^\infty a_k e_k  \,.
	\end{split}	
	\end{equation}
Due to \eqref{eq_frame} and the general fact that $\norm{F} = \norm{F^*} $, there holds
	\begin{equation}\label{eq_bound_F_Fadj}
		\sqrt{B_1} \leq \norm{F} = \norm{F^*} \leq \sqrt{B_2} \,.
	\end{equation}
Furthermore, one can define the operator $S := F^*F$, i.e.,
	\begin{equation*}
		S x := \sum\limits_{k=1}^\infty \spr{x,e_k} e_k \,,
	\end{equation*}
and it follows that $S$ is a bounded linear operator with $B_1 I \leq S \leq B_2 I$, where $I$ denotes the identity operator. Furthermore, $S$ is invertible and $B_2^{-1} I \leq S^{-1} \leq B_1^{-1} I$. Thus, it follows that if one defines $\et_k := S^{-1}e_k$, then the set $\{\et_k\}_{k\in \N}$ also forms a frame over $X$, with frame bounds $B_2^{-1},B_1^{-1}$, which is called the \emph{dual frame} of $\{e_k\}_{k \in \N}$. For the corresponding frame operator
	\begin{equation*}
		\Ft \, : \, X \to \lt(\N) \,, \qquad
		x \mapsto \Kl{\spr{x,\ekt}}_{k \in \N} \,,
	\end{equation*}	
there holds
	\begin{equation}\label{eq_frame_ops}
		\Ft^*F = F^* \Ft = I \,,
		\qquad
		\text{and}
		\qquad
		\Ft  F^* = F \Ft^* = P \,,
	\end{equation}
where $P$ denotes the orthogonal projector from $\lt(\N)$ onto $R(F) = R(\Ft)$. In particular, it follows from \eqref{eq_frame_ops} that any $x \in X$ can be written in the form
	\begin{equation}\label{eq_frame_rec}
		x = \sum\limits_{k=1}^\infty \spr{x, \et_k} e_k
		= \sum\limits_{k=1}^\infty \spr{x, e_k} \et_k \,.
	\end{equation}
Furthermore, if $\{e_k\}_{k \in \N}$ is a tight frame with frame bounds $B_1 = B_2 = B$, then we have that $S^{-1} = B^{-1} I$, which implies that $\et_k = e_k/B$ and therefore,
	\begin{equation}\label{eq_frame_rec_tight}
		x = \frac{1}{B} \sum\limits_{k = 1}^\infty \spr{x, e_k} e_k\,,
		\qquad \text{and} \qquad
		\norm{x}^2 = \frac{1}{B} \sum\limits_{k=1}^\infty \abs{\spr{x,e_k}}^2 \,.
	\end{equation}
Since in general there holds $\Kl{0} \subset \,  N(F^*) = N(\Ft^*)$, the decomposition of $x$ given in \eqref{eq_frame_rec} is not unique. However, it is the most \emph{economical one} in the sense of the following
\begin{proposition}[{\hspace{-1pt}\cite[Prop.~3.2.4]{Daubechies_1992}}]\label{prop_frame_min}
	If $x = \sum\limits_{k=1}^\infty a_k e_k$ for some sequence $\{a_k\}_{k\in \N} \in \lt(\N)$ and if $a_k\neq\spr{x,\et_k}$ for some $k \in \N$, then
	\begin{equation*}
	\sum\limits_{k=1}^\infty \abs{a_k}^2 > \sum\limits_{k = 1} ^\infty \abs{\spr{x,\et_k}}^2 \,.
	\end{equation*}
\end{proposition}

This proposition implies that among all possible decompositions of a function $x$ in terms of the frame $\Kl{e_k}_{k\in\N}$, the coefficients $\spr{x,\ekt}$ in \eqref{eq_frame_rec} are those with the smallest $\ell_2$-norm. Thus, working with these coefficients also has the practical advantage of reducing the risk of computational instabilities and the resulting errors in an implementation.

The fact that frames generally allow the decomposition of a function in potentially infinitely many different ways is one of the key differences between frames and bases. In fact, for any frame $\Kl{e_k}_{k\in\N}$ the following statements are equivalent (see e.g.~\cite{Christensen_2016}):
	\begin{equation}\label{eq_frame_equivalences}
	\begin{split}
		N(F^*) = 0 
		\quad &\Leftrightarrow \quad
		\Kl{e_k}_{k\in\N} \text{ is a (Riesz) basis}
		\\
		\quad &\Leftrightarrow \quad
		\Kl{e_k}_{k\in\N} \text{ and } \Kl{\ekt}_{k\in\N} \text{ are biorthogonal}
		\\
		\quad &\Leftrightarrow \quad
		\Kl{e_k}_{k\in\N} \text{ is exact (i.e.\ no element can be deleted)} \,.
	\end{split}
	\end{equation}
In particular, since $R(F) = R(\Ft)$ is a closed subspace of $\lt(\N)$ and thus there holds $\lt(\N) = R(F) \oplus N(F^*)$, it follows that $P=I$ if and only if $\Kl{e_k}_{k\in\N}$ is an exact frame.

Note that it is sometimes not possible to compute the dual frame function $\et_k$ explicitly. However, since there holds (see \cite{Daubechies_1992}) that
	\begin{equation}\label{eq_dual}
		\et_k = \frac{2}{B_1+B_2} \sum\limits_{j=0}^{\infty} R^j e_j \,,
	\end{equation}
where $R := I - \frac{2}{B_1+B_2} S$, the elements of the dual frame can be approximated by only summing up to a finite index $N$, i.e.,
	\begin{equation}\label{eq_dual_N}
		\et_k^N = \frac{2}{B_1+B_2} \sum\limits_{j=0}^{N} R^j e_j \,.
	\end{equation}
The induced error of this approximation is controlled by the frame bounds $B_1,B_2$, i.e.,
	\begin{equation}\label{eq_dual_approx_error}
		\norm{x - \sum\limits_{k=1}^\infty \spr{x,e_k}\et_k^N} \leq \kl{\frac{B_2-B_1}{B_2+B_1}}^{N+1}\norm{x}\,.
	\end{equation}
Note that \eqref{eq_dual_N} can also be written in the recursive form
	\begin{equation}\label{eq_dual_N_approx}
		\et_k^N = \frac{2}{B_1+B_2}e_k + R \et_k ^{N-1} \,,
	\end{equation}
which allows for an efficient numerical implementation.

\section{Frame Decomposition}\label{sect_frame_dec}

In this section, we derive our main results on the frame decomposition of bounded linear operators $A \, : \, X \to Y$, where $X$ and $Y$ are real or complex Hilbert spaces of the general form \eqref{def_X_Y} for some $M,N\in\N$, i.e., 
	\begin{equation}\label{def_A}
	\begin{split}
		A \, : \, X = \prod\limits_{m=1}^M X_m 
		\, &\to \,,
		Y = \prod\limits_{n=1}^N Y_n 
		\\
		x = (x_m)_{m=1}^M \, &\mapsto \, Ax = (A_n x)_{n=1}^N \,,
	\end{split}
	\end{equation}
where the operators $A_n : X \to Y_n$ simply denote the components of the operator $A$. On the spaces $X$ and $Y$, we consider the canonic inner products 
	\begin{equation*}
		\spr{x,z}_X := \sum_{m=1}^{M} \spr{x_m,z_m}_{X_m} \,,
		\qquad 
		\spr{w,y}_Y := \sum_{n=1}^{N} \spr{w_n,y_n}_{Y_n} \,,
	\end{equation*}
where $\spr{\cdot,\cdot}_{X_m}$ and $\spr{\cdot,\cdot}_{Y_n}$ denote the inner products of $X_m$ and $Y_n$, respectively. 

For each of the spaces $X_m$ and $Y_n$, we consider frames $\{e_k^m\}_{k\in\N}$ and $\{f_k^n\}_{k\in\N}$ with frame bounds $B_1, B_2$ and $C_1,C_2$, respectively, i.e., for all $x_m \in X_m$ there holds \\
	\begin{equation}\label{eq_frame_ek}
		B_1 \norm{x_m}_{X_m}^2 \leq \sum\limits_{k=1}^\infty \abs{\spr{x_m,e_k^m}_{X_m}}^2 \leq B_2 \norm{x_m}_{X_m}^2 \,, 
	\end{equation}
and for all $y_n \in Y_n$ there holds
	\begin{equation}\label{eq_frame_fk}
		C_1 \norm{y_n}_{Y_n}^2 \leq \sum\limits_{k=1}^\infty \abs{\spr{y_n,f_k^n}_{Y_n}}^2 \leq C_2 \norm{y_n}_{Y_n}^2 \,.
	\end{equation}
For our analysis, these frames have to be suitably connected to the operator $A$, which leads us to 
\begin{assumption}\label{ass_main}
There exists a sequence $\{\Lk \}_{k\in\N}$ of complex-valued $N\times M$ matrices, i.e., $\Lk = (\lkmn)_{m,n=1}^{M,N}$ with $\lkmn \in \C$, such that
	\begin{equation}\label{cond_frames}
		\kl{\spr{A_n x,f_k^n}_{Y_n}}_{n=1}^N = \Lk \cdot \kl{\spr{x_m,e_k^m}_{X_m}}_{m=1}^M \,,
		\qquad 
		\forall \,k \in \N \,, \,  x = (x_m)_{m=1}^M \in X \,,
	\end{equation}
where as in \eqref{def_A} the operators $A_n : X \to Y_n$ denote the components of the operator $A$.
\end{assumption}

The implications of this assumption and the question of how to choose suitable frames $\Kl{e_k^m}_{k\in \N}$ and $\Kl{f_k^n}_{k\in\N}$ are discussed in detail in Section~\ref{sect_discussion} below. Assumption~\ref{ass_main} immediately leads to

\begin{lemma}\label{lemma_An}
Let $A$ be as in \eqref{def_A} and let Assumption~\ref{ass_main} hold. Then there holds
	\begin{equation}\label{eq_lemma_An_helper}
		\spr{A_n x, f_k^n}_{Y_n} = \sum_{m=1}^M  \lkmn \spr{x_m,e_k^m}_{X_m} \,,
		\qquad \forall \, x = (x_m)_{m=1}^M \in X \,.
	\end{equation} 
\end{lemma}

Next, we want to derive a decomposition of $A$ in terms of the functions $e_k^m$ and $f_k^n$. For this, we need to consider the dual frames $\{\ekt^m\}_{k\in\N}$ and $\{\fkt^n\}_{k\in\N}$ of the frames  $\{e_k^m\}_{k\in\N}$ and $\{f_k^n\}_{k\in\N}$, respectively. It follows from Section~\ref{sect_Frames} together with \eqref{eq_frame_ek} and \eqref{eq_frame_fk} that these dual frames are again frames, but now with frame bounds $1/B_2,1/B_1$ and $1/C_2, 1/C_1$, respectively. In particular, this means that any function $x_m \in X_m$ and $y_n \in Y_n$ can be written in the form
	\begin{equation}\label{eq_dec_x_y}
		x_m = \sum_{k=1}^\infty \spr{x_m,e_k^m}_{X_m} \ekt^m \,,
		\qquad
		\text{and}
		\qquad
		y_n = \sum_{k=1}^\infty \spr{y_n,f_k^n}_{Y_n} \fkt^n \,,
	\end{equation}
respectively, and that these decompositions are the most economical ones in the sense of Proposition~\ref{prop_frame_min}. Using this, we now obtain the following decomposition result:

\begin{proposition}\label{prop_dec_A}
Let $A$ be as in \eqref{def_A} and let Assumption~\ref{ass_main} hold. Then there holds
	\begin{equation}\label{eq_A_dec}
		A x = \sum_{k=1}^\infty  \kl{\sum_{m=1}^M \lkmn \spr{x_m,e_k^m}_{X_m} \fkt^n }_{n=1}^N  \,,
		\qquad \forall \, x = (x_m)_{m=1}^M \in X \,.
	\end{equation}
\end{proposition}
\begin{proof}
Since $\{f_k^n\}_{k\in\N}$ forms a frame over $Y_n$, it follows from \eqref{eq_dec_x_y} that for all $x \in X$,
	\begin{equation*}
		A_n x = \sum_{k=1}^\infty \spr{A_n x,f_k^n}_{Y_n} \fkt^n \,,
	\end{equation*}
Together with the definition of $A$ and Lemma~\ref{lemma_An}, this implies
	\begin{equation*}
		Ax = (A_n x)_{n=1}^N = \kl{ \sum_{k=1}^\infty \spr{A_n x,f_k^n}_{Y_n} \fkt^n }_{n=1}^N
		\overset{\eqref{eq_lemma_An_helper}}{=}
		\kl{ \sum_{k=1}^\infty \sum_{m=1}^M \lkmn \spr{x_m,e_k^m}_{X_m} \fkt^n }_{n=1}^N \,,
	\end{equation*}
which yields the assertion.
\end{proof}

Using the decomposition of $A$ derived above, we are now turning our attention to the task of obtaining a suitable solution of the operator equation \eqref{Ax=y}. For this, we first need to derive the following auxiliary result:
\begin{lemma}\label{lem_residual}
Let $A$ be as in \eqref{def_A} and let Assumption~\ref{ass_main} hold. Then for all functions $x = (x_m)_{m=1}^M \in X$ and $y = (y_n)_{n=1}^N \in Y$ there holds
	\begin{equation}\label{eq_residual}
		C_1 \norm{Ax - y}_Y^2 
		\leq \ \sum_{k=1}^\infty
		\norm{ \Lk \cdot \kl{\spr{x_m,e_k^m}_{X_m}}_{m=1}^M - \kl{ \spr{y_n,f_k^n}_{Y_n}}_{n=1}^N }_{\R^N}^2  
		\leq 
		C_2 \norm{Ax - y}_Y^2 \,.
	\end{equation}
where $C_1$ and $C_2$ denote the frame bounds of $\Kl{f_k^n}_{k\in\N}$. 
\end{lemma}
\begin{proof}
Let $x = (x_m)_{m=1}^M \in X$ and $y = (y_n)_{n=1}^N \in Y$ be arbitrary but fixed. Since the sets $\{f_k^n\}_{k\in\N}$ form frames over $Y_n$ with frame bounds $C_1, C_2$, it follows from \eqref{eq_frame_fk} that 
	\begin{equation*}
	\begin{split}
		C_1 \norm{Ax-y}_Y^2 
		\leq 
		\sum\limits_{n=1}^N \sum\limits_{k=1}^\infty \abs{\spr{A_n x-y_n,f_k^n}_{Y_n}}^2 
		\leq 
		C_2 \norm{Ax-y}_Y^2 \,.	
	\end{split}
	\end{equation*}
Together with the fact that Lemma~\ref{lemma_An} and $\Lk = (\lkmn)_{m,n=1}^{M,N}$ imply 
	\begin{equation*}
	\begin{split}
		\sum_{k=1}^\infty \abs{\spr{A_nx - y_n, f_k^n}_{Y_n}}^2
		&=
		\sum_{k=1}^\infty \sum_{n=1}^N \abs{\sum_{m=1}^M \lkmn \spr{x_m,e_k^m}_{X_m}  -   \spr{y_n, f_k^n}_Y}^2 
		\\
		&= \sum_{k=1}^\infty
		\norm{ \Lk \cdot \kl{\spr{x_m,e_k^m}_{X_m}}_{m=1}^M - \kl{ \spr{y_n,f_k^n}_{Y_n}}_{n=1}^N }_{\R^N}^2
		\,,
	\end{split}
	\end{equation*}	
this yields the assertion.
\end{proof}

The above result has many useful consequences, such as the following	
\begin{corollary}\label{corollary_equivalences}
Let $A$ be as in \eqref{def_A} and let Assumption~\ref{ass_main} hold. Then for any function $x = (x_m)_{m=1}^M \in X$ there holds
	\begin{equation}\label{equivalence_nullspace}
		x \in N(A) 
		\qquad
		\Longleftrightarrow
		\qquad
		\forall \, k \in \N  : \quad 
		\kl{\spr{x_m,e_k^m}_{X_m}}_{m=1}^M \in N(\Lk) \,.
	\end{equation}
Furthermore, for any $y = (y_n)_{n=1}^N \in Y$ there holds
	\begin{equation}\label{equivalence_solution}
		A x = y
		\qquad
		\Longleftrightarrow
		\qquad
		\forall \, k \in \N  : \quad 
		\Lk \cdot \kl{\spr{x_m,e_k^m}_{X_m}}_{m=1}^M = \kl{ \spr{y_n,f_k^n}_{Y_n}}_{n=1}^N \,.
	\end{equation}	
\end{corollary} 
\begin{proof}
Both equivalences follow directly from Lemma~\ref{lem_residual}.
\end{proof}

The above corollary, and in particular equation \eqref{equivalence_solution}, point to a strong connection between the structure of the equation $Ax=y$ and the properties of the matrices $\Lk$, which we now analyse in more detail. We first consider the singular-value decomposition of each of the matrices $\Lk$, i.e., the vectors $\vkj \in \C^M$, $\ukj \in \C^N$, and $\mukj \in \R^+$, for $j=1,\dots,\rk $, with $\rk \leq \min\Kl{M,N}$ denoting the rank of $\Lk$, such that
	\begin{equation}\label{eq_SVD_sigma}
	\begin{split}
		\Lk \cdot \boldsymbol{h}  = &\sum\limits_{j=1}^{\rk} \mukj \kl{\vkj^H \,\cdot\, \boldsymbol{h} } \ukj \,,
		\\
		\vkj^H \,\cdot\, \vkl =& \,\, \delta_{jl} \,,
		\quad
		\ukj^H \,\cdot\, \ukl = \delta_{jl} \,,
		\\
		\mu_{k,1} &\geq \dots \geq \mu_{k,\rk} > 0 \,.	
	\end{split}
	\end{equation}
Here, the superscript $H$ denotes the Hermitian, i.e., the complex-conjugate transpose, of a vector. We collect the singular values and vectors into the singular systems $(\mukj,\ukj,\vkj)_{j=1}^{r_k}$. Note that the singular vectors $\vkj$ are eigenvectors of the matrices $\Lk^H \cdot \Lk$ and form bases of $N(\Lk)^\perp \subseteq \C^M$, and similarly, the singular vectors $\ukj$ are eigenvectors of the matrices $\Lk \cdot \Lk^H$ and form bases of $R(\Lk) \subseteq \C^N$. Furthermore, for any $\boldsymbol{w} \in \C^N$ the unique minimizer of minimum-norm of the functional
	\begin{equation*}
		\boldsymbol{h} \mapsto \norm{\Lk \cdot \boldsymbol{h} - \boldsymbol{w}}_{\R^N} \,,
	\end{equation*} 
is given by $\Lk^\dagger \cdot \boldsymbol{w}$, where $\Lk^\dagger$ denotes the \emph{pseudo-inverse} of $\Lk$ defined via
	\begin{equation}\label{eq_Lk_dagger}
		\Lk^\dagger \cdot \boldsymbol{w} := \sum\limits_{j=1}^{\rk} \frac{1}{\mukj} \kl{\ukj^H \,\cdot\, \boldsymbol{w}} \vkj \,. 
	\end{equation}

Using the singular-systems and the pseudo-inverses of the matrices $\Lk$ we now make
\begin{definition}
Let $A$ be as in \eqref{def_A} and let Assumption~\ref{ass_main} hold. Furthermore, let $(\mukj,\ukj,\vkj)_{j=1}^{r_k}$ be the singular systems of the matrices $\Lk$ as defined in \eqref{eq_SVD_sigma}. Then for $y = (y_n)_{n=1}^N \in Y$ we define 
	\begin{equation}\label{def_AD}
	\begin{split}
		\AD y &:= \kl{\Ft_m^* \Kl{\kl{ \Lk^\dagger \cdot \kl{\spr{y_n,f_k^n}_{Y_n}}_{n=1}^N}_m  }_{k\in\N}}_{m=1}^M  \,,
	\end{split}
	\end{equation}
where $\Ft_m$ denotes the frame operator corresponding to the dual frame $\Kl{\ekt^m}_{k\in\N}$.
\end{definition}

Note that due to \eqref{def_F_Fad} and \eqref{eq_Lk_dagger} the above definition of $\AD y$ is equivalent to
	\begin{equation}\label{def_AD_expl}
	\begin{split}
		\AD y  & \,\overset{\eqref{def_F_Fad}}{=}\sum_{k=1}^\infty \kl{   \kl{ \Lk^\dagger \cdot \kl{\spr{y_n,f_k^n}_{Y_n}}_{n=1}^N}_m  \ekt^m  }_{m=1}^M
		\\ &\overset{\eqref{eq_Lk_dagger}}{=}
		\sum_{k=1}^\infty \sum\limits_{j=1}^{\rk}\frac{1}{\mukj}\kl{\ukj^H \,\cdot \kl{ \spr{y_n,f_k^n}_{Y_n}}_{n=1}^N   }\kl{   \kl{  \vkj }_m \, \ekt^m }_{m=1}^M \,.
	\end{split}
	\end{equation}
Concerning the well-definedness of the operator $\AD$, we have the following

\begin{lemma}\label{lemma_Picard}
Let $y = (y_n)_{n=1}^N \in Y$ and let $\AD y$ be defined as in \eqref{def_AD}. Then $\AD y$ is a well-defined element of $X$, i.e., $\norm{\AD y}_X < \infty$, if the following \emph{Picard condition} holds:
	\begin{equation}\label{cond_Picard}
		\sum\limits_{k = 1}^\infty \sum\limits_{j=1}^{\rk}\frac{1}{\mukj^2}\abs{ \kl{ \spr{y_n,f_k^n}_{Y_n}}_{n=1}^N }^2 
		< 
		\infty \,.
	\end{equation}
\end{lemma}
\begin{proof}
Since the sets $\{e_k^m \}_{k\in\N}$ form frames over $X_m$ with frame bounds $B_1,B_1$, the dual frames $\{\ekt^m\}_{k\in\N}$ form frames over $X_m$ with frame bounds $1/B_2,1/B_1$, and thus with
	\begin{equation}\label{def_helper_am}
		a^m := \Kl{ \kl{\Lk^\dagger \cdot 	\kl{\spr{y_n,f_k^n}_{Y_n}}_{n=1}^N}_m }_{k\in\N}  \,,
		\qquad
		\forall \, m \in \Kl{1\,,\dots\,, M} \,,
	\end{equation}
there follows
	\begin{equation}\label{eq_ADy_norm_estimate}
	\begin{split}
		\norm{\AD y}_X^2 
		=
		\sum_{m=1}^{M} \norm{(\AD y)_m}_{X_m}^2  
		\overset{\eqref{def_AD}}{=}
		\sum_{m=1}^{M} 
		\norm{\Ft_m^* \, a^m}_{X_m}^2
		\overset{\eqref{eq_bound_F_Fadj}}{\leq}
		\frac{1}{B_1} 
		\sum_{m=1}^{M} \norm{a^m }_{\ltN}^2 \,.
	\end{split}
	\end{equation}
Next, note that due to \eqref{eq_Lk_dagger} there holds
	\begin{equation}\label{eq_LkD_SVD}
		(a^m)_k = \kl{ \Lk^\dagger \cdot \kl{\spr{y_n,f_k^n}_{Y_n}}_{n=1}^N}_m = \sum\limits_{j=1}^{\rk}\frac{1}{\mukj}\kl{\ukj^H \,\cdot\kl{ \spr{y_n,f_k^n}_{Y_n}}_{n=1}^N   }\kl{  \vkj }_m \,,
	\end{equation}
and thus
	\begin{equation}\label{eq_helper_4}
	\begin{split}
		\sum_{m=1}^{M} \norm{a^m }_{\ltN}^2     
		& \overset{\eqref{eq_LkD_SVD}}{=}
		\sum_{m=1}^{M} \sum\limits_{k = 1}^\infty \abs{\sum\limits_{j=1}^{\rk}\frac{1}{\mukj}\kl{\ukj^H \,\cdot\kl{ \spr{y_n,f_k^n}_{Y_n}}_{n=1}^N   }\kl{  \vkj }_m}^2 
		\\
		&=
		 \sum\limits_{k = 1}^\infty \norm{\sum\limits_{j=1}^{\rk}\frac{1}{\mukj}\kl{\ukj^H \,\cdot \kl{ \spr{y_n,f_k^n}_{Y_n}}_{n=1}^N } \vkj }_{\R^M}^2 \,.
	\end{split}
	\end{equation}
Furthermore, due to the orthonormality of the singular vectors $\ukj$ and $\vkj$ there holds
	\begin{equation*}
	\begin{split}
		\sum\limits_{k = 1}^\infty \norm{\sum\limits_{j=1}^{\rk}\frac{1}{\mukj}\kl{\ukj^H \,\cdot \kl{ \spr{y_n,f_k^n}_{Y_n}}_{n=1}^N } \vkj }_{\R^M}^2 
		=
		\sum\limits_{k = 1}^\infty \sum\limits_{j=1}^{\rk}\frac{1}{\mukj^2}\abs{ \kl{ \spr{y_n,f_k^n}_{Y_n}}_{n=1}^N }^2\,,
	\end{split}
	\end{equation*}
which together with \eqref{eq_helper_4} yields
	\begin{equation}\label{ineq_helper_am}
	\begin{split}
		\sum_{m=1}^{M} \norm{a^m }_{\ltN}^2
		=
		\sum\limits_{k = 1}^\infty \sum\limits_{j=1}^{\rk}\frac{1}{\mukj^2}\abs{ \kl{ \spr{y_n,f_k^n}_{Y_n}}_{n=1}^N }^2\,.
	\end{split}
	\end{equation}
Hence, together with \eqref{eq_ADy_norm_estimate} we obtain that 
	\begin{equation*}
		\norm{\AD y}_X^2 
		\overset{\eqref{eq_ADy_norm_estimate}}{\leq}
		\frac{1}{B_1} 
		\sum_{m=1}^{M} \norm{a^m }_{\ltN}^2
		\overset{\eqref{ineq_helper_am}}{=} \,\,
		\frac{1}{B_1}  
		\sum\limits_{k = 1}^\infty \sum\limits_{j=1}^{\rk}\frac{1}{\mukj^2}\abs{ \kl{ \spr{y_n,f_k^n}_{Y_n}}_{n=1}^N }^2 \,,
	\end{equation*}
which yields the assertion.
\end{proof}

We are now able to derive the first main result of this paper:
\begin{theorem}\label{thm_main_I}
Let $A$ be as in \eqref{def_A}, let Assumption~\ref{ass_main} hold, and let $y \in R(A)$. Furthermore, assume that $N(\Lk) = \Kl{0}$ and let $\AD y$ be defined as in \eqref{def_AD}. Then $\AD y $ is a well-defined element of $X$ and the unique solution of \eqref{Ax=y}. Additionally, among all possible decompositions of $\AD y$ in terms of the dual frame functions $\ekt^m$, the decomposition~\eqref{def_AD} is the most economical one in the sense of Proposition~\ref{prop_frame_min}.
\end{theorem}
\begin{proof}
Let $y \in R(A)$ be arbitrary but fixed. Since $y \in R(A)$ there exists a function $\xb \in X$ such that $A \xb = y$. Hence, due to \eqref{equivalence_solution} it follows that for all $k\in\N$, the expansion coefficients $\spr{\xb_m,e_k^m}_{X_m}$ of $\xb$ are solutions of the matrix-vector systems
	\begin{equation*}
		\Lk \cdot \kl{\spr{\xb_m,e_k^m}_{X_m}}_{m=1}^M = \kl{ \spr{y_n,f_k^n}_{Y_n}}_{n=1}^N \,.
	\end{equation*} 	
Since $N(\Lk) = \Kl{0}$, it follows that
	\begin{equation}\label{eq_xb_coeffs}
		\kl{\spr{\xb_m,e_k^m}_{X_m}}_{m=1}^M
		=	\Lk^\dagger \cdot \kl{ \spr{y_n,f_k^n}_{Y_n}}_{n=1}^N
		\,,
	\end{equation}
and thus, via the reconstruction formula \eqref{eq_dec_x_y} we obtain that
	\begin{equation}\label{eq_helper_1}
	\begin{split}
		\xb = \kl{ \sum_{k=1}^\infty  \kl{\Lk^\dagger \cdot \kl{ \spr{y_n,f_k^n}_{Y_n}}_{n=1}^N}_{m} \ekt^m }_{m=1}^M \,.
	\end{split}
	\end{equation}
Comparing \eqref{eq_helper_1} with \eqref{def_AD} and \eqref{def_AD_expl} we conclude that $\xb = \AD y$, and thus $\AD y$ is both a well-defined element of $X$ and a solution of \eqref{Ax=y}. Furthermore, since
	\begin{equation*}
		\spr{(\AD y)_m, e_k^m}_{X_m} = \spr{\xb_m, e_k^m}_{X_m}  \overset{\eqref{eq_xb_coeffs}}{=} 	\kl{\Lk^\dagger \cdot \kl{ \spr{y_n,f_k^n}_{Y_n}}_{n=1}^N }_m\,,
	\end{equation*}
it follows from Proposition~\ref{prop_frame_min} that \eqref{def_AD} is the most economic decomposition of $\xb = \AD y$ in terms of the dual frame functions $\ekt^m$. Finally, since we assumed that $N(\Lk) = 0$, it follows from \eqref{equivalence_nullspace} that $N(A) = 0$ and thus that the operator $A$ is injective. Hence, $\AD y$ is the unique solution of \eqref{Ax=y}, which concludes the proof.
\end{proof}

\vspace{0pt}
\begin{remark}
In Theorem~\ref{thm_main_I} we assumed that $N(\Lk) = \Kl{0}$, which implied the injectivity of the operator $A$ and thus the uniqueness of a solution of \eqref{Ax=y}. On the other hand, if $N(\Lk) \neq \Kl{0}$ then due to \eqref{equivalence_solution}, for any solution $x = (x_m)_{m=1}^M \in X$ of \eqref{Ax=y} there holds
	\begin{equation}\label{eq_projected_solution}
		P_{N(\Lk)^\perp}   \kl{\spr{x_m,e_k^m}_{X_m}}_{m=1}^M  =  \Lk^\dagger \cdot \kl{\spr{y_n,f_k^n}_{Y_n}}_{n=1}^N \,,
		\qquad
		\forall \, k \in \N \,.
	\end{equation}
Now, since for $y \in R(A)$ the full solution set of \eqref{Ax=y} is given by $\xd + N(A)$, there exists an element $\xb \in N(A)$ such that $x = \xd + \xb$. In particular, due to \eqref{equivalence_nullspace} there holds
	\begin{equation*}
		\kl{\spr{\xb_m ,e_k^m }_{X_m}}_{m=1}^M \in N(\Lk) \,, 
		\qquad 
		\forall \, k \in \N \,.
	\end{equation*}
Hence, it follows from \eqref{eq_projected_solution} that 
	\begin{equation*}
		P_{N(\Lk)^\perp}   \kl{\spr{x^\dagger_m,e_k^m}_{X_m}}_{m=1}^M  =  \Lk^\dagger \cdot \kl{\spr{y_n,f_k^n}_{Y_n}}_{n=1}^N \,,
		\qquad 
		\forall \, k \in \N \,,
	\end{equation*}
which can be rewritten as		
	\begin{equation*}
	\begin{split}
		\kl{\spr{x^\dagger_m,e_k^m}_{X_m}}_{m=1}^M  
		&=  \Lk^\dagger \cdot \kl{\spr{y_n,f_k^n}_{Y_n}}_{n=1}^N  + (I - P_{N(\Lk)^\perp})   \kl{\spr{x^\dagger_m,e_k^m}_{X_m}}_{m=1}^M 
		\\
		&= \Lk^\dagger \cdot \kl{\spr{y_n,f_k^n}_{Y_n}}_{n=1}^N  + P_{N(\Lk)}   \kl{\spr{x^\dagger_m,e_k^m}_{X_m}}_{m=1}^M \,.
	\end{split}
	\end{equation*}
Now by the definition \eqref{def_AD} of $\AD y$ there holds
	\begin{equation*}
		\kl{\AD y }_m = \Ft_m^* \Kl{ \kl{ \Lk^\dagger \cdot \kl{\spr{y_n,f_k^n}_{Y_n}}_{n=1}^N }_m }_{k\in\N} \,,
	\end{equation*}
and thus
	\begin{equation*}
		\xd  = \AD y 
		+ \kl{ \Ft_m^* a^m }_{m=1}^M \,,
		\qquad
		a^m = \Kl{  \kl{ P_{N(\Lk)}   \kl{\spr{x^\dagger_\mb,e_k^\mb}_{X_\mb}}_{\mb=1}^M }_m }_{k\in\N} \,.
	\end{equation*}
Hence, we see that the frame coefficients of $\xd$ and the nullspaces $N(\Lk)$ directly determine the distance of $\AD y$ to the minimum-norm solution $\xd$.
\end{remark}\vspace{7pt}

Next, we proceed to derive the second main result of this paper:

\begin{theorem}\label{thm_main_II}
Let $A$ be as in \eqref{def_A}, let Assumption~\ref{ass_main} hold, and let  $y = (y_n)_{n=1}^N \in Y$. Furthermore, let the frames $\{f_k^n\}_{k\in\N}$ be tight and assume that
	\begin{equation}\label{condition_Range}
		\Kl{ \kl{\Lk^\dagger \cdot \kl{\spr{y_n,f_k^n}_{Y_n}}_{n=1}^N}_m }_{k\in\N} \in R(F_m)  \,,
		\qquad
		\forall \, m \in \Kl{1\,,\dots\,, M} \,,
	\end{equation}
where $F_m$ denotes the frame operator corresponding to the frame $\Kl{e_k^m}$. Then $\AD y$ as given in \eqref{def_AD} is a well-defined element of $X$ and a minimum-coefficient least-squares solution of equation \eqref{Ax=y}, i.e., it is a least-squares solution of \eqref{Ax=y} satisfying
	\begin{equation}\label{eq_AD_min}
		\sum\limits_{m=1}^M \abs{\spr{(\AD y)_m , e_k^m}_{X_m} }^2
		\leq
		\sum\limits_{m=1}^M \abs{\spr{x^*_m , e_k^m}_{X_m} }^2 \,,
	\end{equation}
for any least-squares solution $x^* = (x^*_m)_{m=1}^M  \in X$ of \eqref{Ax=y} and any $k \in \N$. Furthermore,  
	\begin{equation}\label{eq_distance_ADy_xD}
		\norm{\xd}_X \leq \norm{\AD y}_X \leq \sqrt{B_2/B_1} \norm{\xd}_X \,,
	\end{equation}
where $B_1$ and $B_2$ are the frame bounds of $\{e_k^m\}_{k\in \N}$. Hence, if also the frames $\{e_k^m\}_{k\in \N}$ are tight, then $\AD y$ coincides with the minimum-norm least-squares solution $\xd$. 

On the other hand, if additionally there holds
	\begin{equation}\label{cond_range_solvable}
		\kl{ \spr{y_n,f_k^n}_{Y_n}}_{n=1}^N \in R(\Lk)  \,, 
		\qquad
		\forall \, k \in \N \,,
	\end{equation}
then $\AD y$ is also a solution of \eqref{Ax=y}. Finally, among all possible decompositions of $\AD y$ in terms of the dual frame functions $\ekt^m$, the decomposition~\eqref{def_AD} is the most economical one in the sense of Proposition~\ref{prop_frame_min}.
\end{theorem}
\begin{proof}
Let $y = (y_n)_{n=1}^N \in Y$ be arbitrary but fixed and let $\AD y$ be given as in \eqref{def_AD}, i.e., 
	\begin{equation*}
		\AD y := \kl{\Ft_m^* \Kl{\kl{ \Lk^\dagger \cdot \kl{\spr{y_n,f_k^n}_{Y_n}}_{n=1}^N}_m  }_{k\in\N}}_{m=1}^M  \,.
	\end{equation*}
First of all, note that due to \eqref{eq_ADy_norm_estimate} there holds	
	\begin{equation*}
		\norm{\AD y}_X^2 \leq \frac{1}{B_1} 
		\sum_{m=1}^{M} \norm{\Kl{\kl{ \Lk^\dagger \cdot \kl{\spr{y_n,f_k^n}_{Y_n}}_{n=1}^N}_m  }_{k\in\N}}_{\ltN}^2 \,,
	\end{equation*}	
Hence, since we assumed that \eqref{condition_Range} holds and since $R(F_m) \subset \ltN$, it follows that $\AD y$ is a well-defined element of $X$. Furthermore, since we have
	\begin{equation*}
		\Kl{ \spr{(\AD y)_m , e_k^m }_{X_m} }_{k\in\N} = F_m \kl{(\AD y)_m } 
		= F_m \Ft_m^* \Kl{\kl{ \Lk^\dagger \cdot \kl{\spr{y_n,f_k^n}_{Y_n}}_{n=1}^N}_m }_{k\in\N} \,,  
	\end{equation*} 
and since, with $P_m$ denoting the orthogonal projector onto $R(F_m)$, there holds
	\begin{equation*}
	\begin{split}
		F_m \Ft_m^* \Kl{\kl{ \Lk^\dagger \cdot \kl{\spr{y_n,f_k^n}_{Y_n}}_{n=1}^N}_m }_{k\in\N} 
		&\overset{\eqref{eq_frame_ops}}{=}
		P_{m} \Kl{\kl{ \Lk^\dagger \cdot \kl{\spr{y_n,f_k^n}_{Y_n}}_{n=1}^N}_m }_{k\in\N}
		\\ 
		&\overset{\eqref{condition_Range}}{=} \Kl{\kl{ \Lk^\dagger \cdot \kl{\spr{y_n,f_k^n}_{Y_n}}_{n=1}^N}_m }_{k\in\N} \,,
	\end{split}
	\end{equation*}
it follows that
	\begin{equation}\label{AD_canonic_expansion_coefficients}
		\spr{(\AD y)_m , e_k^m }_{X_m}  = \kl{ \Lk^\dagger \cdot \kl{\spr{y_n,f_k^n}_{Y_n}}_{n=1}^N}_m  \,.  
	\end{equation} 
Now, by assumption the sets $\{f_k^n\}_{k\in\N}$ form tight frames over the spaces $Y_n$, i.e., they satisfy \eqref{eq_frame_fk} with $C_1 = C_2 = C$ for some $C > 0$. Hence, it follows from \eqref{eq_residual} that
	\begin{equation}\label{eq_residual_tight}
		\norm{Ax - y}_Y^2 
		= \frac{1}{C} \sum_{k=1}^\infty
		\norm{ \Lk \cdot \kl{\spr{x_m,e_k^m}_{X_m}}_{m=1}^M - \kl{ \spr{y_n,f_k^n}_{Y_n}}_{n=1}^N }_{\R^N}^2  \,.
	\end{equation}
Hence, if $x = (x_m)_{m=1}^M\in X$ can be found such that for each $k \in \N$ its expansion coefficients $\spr{x_m,e_k^m}_{X_m}$ minimize the expressions 
	\begin{equation}\label{min_problems}
		\norm{ \Lk \cdot \kl{\spr{x_m,e_k^m}_{X_m}}_{m=1}^M - \kl{ \spr{y_n,f_k^n}_{Y_n}}_{n=1}^N }_{\R^N}  \,,
	\end{equation}
then $x$ is a minimizer of $\norm{Ax - y}_Y$ and thus a least-squares solution of \eqref{Ax=y}. Due to \eqref{AD_canonic_expansion_coefficients} and the properties of the pseudo-inverse $\Lk^\dagger$ of $\Lk$, this is exactly satisfied for the choice $x = \AD y$. Hence, it follows that $\AD y$ is a least-squares solution of \eqref{Ax=y}. Furthermore, it follows that any other least-squares solution $x^*$ of \eqref{Ax=y} also has to minimize each of the expressions \eqref{min_problems}, and thus is of the form
	\begin{equation*}
		\kl{\spr{x^*_m,e_k^m}_{X_m}}_{m=1}^M = \Lk^\dagger \cdot \kl{\spr{y_n,f_k^n}_{Y_n}}_{n=1}^N + \bar{z}_k \,,
		\qquad 
		\forall \, k \in \N \,,
	\end{equation*}
for some vectors $\bar{z}_k \in N(\Lk)$. In particular, we have the orthogonality relation
	\begin{equation}\label{ineq_helper_1}
		\norm{\kl{\spr{x^*_m,e_k^m}_{X_m}}_{m=1}^M}_{\R^M}^2 = \norm{\Lk^\dagger \cdot \kl{\spr{y_n,f_k^n}_{Y_n}}_{n=1}^N }_{\R^M}^2 + \norm{\bar{z}_k}_{\R^M}^2 \,,
		\qquad 
		\forall \, k \in \N \,,
	\end{equation}
which follows from the properties of the pseudo-inverse. Using this, we get
	\begin{equation*}
	\begin{split}
		\sum_{m=1}^{M} \abs{\spr{(\AD y)_m,e_k^m}_{X_m}}^2 
		&= 
		\norm{ \kl{\spr{(\AD y)_m,e_k^m}_{X_m}}_{m=1}^M }_{\R^M}^2 
		\overset{\eqref{AD_canonic_expansion_coefficients}}{=}
		\norm{ \Lk^\dagger \cdot \kl{\spr{y_n,f_k^n}_{Y_n}}_{n=1}^N }_{\R^M}^2 
		\\
		&
		\overset{\eqref{ineq_helper_1}}{\leq}
		\norm{\kl{\spr{x^*_m,e_k^m}_{X_m}}_{m=1}^M }_{\R^M}^2
		= 
		\sum_{m=1}^{M} \abs{\spr{x^*_m,e_k^m}_{X_m}}^2 \,,
	\end{split}
	\end{equation*}
which yields \eqref{eq_AD_min}. If in addition \eqref{cond_range_solvable} is satisfied, then due to \eqref{AD_canonic_expansion_coefficients} and the properties of the pseudo-inverse there holds
	\begin{equation*}
		\Lk \cdot \kl{\spr{(\AD y)_m,e_k^m}_{X_m}}_{m=1}^M = \kl{ \spr{y_n,f_k^n}_{Y_n}}_{n=1}^N  \,, 
		\qquad
		\forall \, k \in \N \,,
	\end{equation*}
and thus it follows by choosing $x = \AD y$ in \eqref{eq_residual_tight} that $\AD y$ is also a solution of \eqref{Ax=y}. 
\\ \indent
Next, note that since the sets $\{e_k^m\}_{k\in\N}$ form frames over $X_m$ with frame bounds $B_1,B_2$, it follows from \eqref{eq_frame_ek} that for any $x = (x_m)_{m=1}^M \in X$ there holds 
	\begin{equation}\label{helper_xframe}
		B_1 \norm{x}_{X}^2 \leq 
		\sum_{m=1}^M \sum\limits_{k = 1}^\infty \abs{\spr{x_m,e_k^m}}^2 
		= 
		\sum\limits_{k = 1}^\infty\norm{ \kl{\spr{x_m,e_k^m}}_{m=1}^M}_{\R^M}^2 
		\leq B_2  \norm{x}_{X}^2 \,.
	\end{equation}	
This holds in particular for the minimum-norm least-squares solution $\xd$, and thus, since we saw above that $\AD y$ is also a least-squares solution, it follows that
	\begin{equation*}
	\begin{split}
		\norm{\xd}_X \leq \norm{\AD y }_X
		&\overset{\eqref{helper_xframe}}{\leq}
		\frac{1}{B_1} 
		\sum\limits_{k = 1}^\infty\norm{ \kl{\spr{(\AD y)_m,e_k^m}}_{m=1}^M}_{\R^M}^2 
		\\
		&\overset{\eqref{eq_AD_min}}{\leq}
		\frac{1}{B_1} 
		\sum\limits_{k = 1}^\infty\norm{ \kl{\spr{\xd_m,e_k^m}}_{m=1}^M}_{\R^M}^2
		\overset{\eqref{helper_xframe}}{\leq}
		\frac{B_2}{B_1} \norm{\xd}_X^2 \,,
	\end{split}
	\end{equation*}
which yields \eqref{eq_distance_ADy_xD}. Hence, if the frames $\{e_k^m\}_{k\in\N}$ are tight, i.e., if $B_1 = B_2 = B$ for some $B > 0$, then it follows from the uniqueness of the minimum-norm solution that $\AD y$ coincides with $\xd$. Finally, note that due to \eqref{AD_canonic_expansion_coefficients} and Proposition~\ref{prop_frame_min} it follows that the decomposition of $\AD y$ given in \eqref{def_AD} is the most economical one in terms of the dual frame functions $\ekt^m$, which concludes the proof. 
\end{proof}

A useful consequence of the above theorem is the following
\begin{corollary}\label{corollary_main_I}
Let $A$ be as in \eqref{def_A}, let  $y \in Y$, and let Assumption~\ref{ass_main} hold.  Furthermore, let the frames $\{f_k^n\}_{k\in\N}$ be tight, the frames $\{e_k^m\}_{k\in\N}$ be exact and assume that the Picard condition \eqref{cond_Picard} holds. Then $\AD y$ as given in \eqref{def_AD} is a well-defined element of $X$ and a minimum-coefficient least-squares solution of equation \eqref{Ax=y}, i.e., it is a least-squares solution of \eqref{Ax=y} satisfying \eqref{eq_AD_min} for any least-squares solution $x^* = (x^*_m)_{m=1}^M \in X$ of \eqref{Ax=y} and any $k \in \N$. Furthermore, there holds \eqref{eq_distance_ADy_xD}, where $B_1$ and $B_2$ are the frame bounds of $\{e_k^m\}_{k\in \N}$. Hence, if also the frames $\{e_k^m\}_{k\in \N}$ are tight, then $\AD y$ coincides with the minimum-norm least-squares solution $\xd$. On the other hand, if additionally \eqref{cond_range_solvable} holds then $\AD y$ is also a solution of \eqref{Ax=y}. Finally, among all possible decompositions of $\AD y$ in terms of the dual frame functions $\ekt^m$, the decomposition~\eqref{def_AD} is the most economical one in the sense of Proposition~\ref{prop_frame_min}.	
\end{corollary}
\begin{proof}
Let $y = (y_n)_{n=1}^N \in Y$ be arbitrary but fixed and consider the sequences 
	\begin{equation*}
		a^m := \Kl{ \kl{\Lk^\dagger \cdot \kl{\spr{y_n,f_k^n}_{Y_n}}_{n=1}^N}_m }_{k\in\N}  \,,
		\qquad
		\forall \, m \in \Kl{1\,,\dots\,, M} \,.
	\end{equation*}
which we already considered in the proof of Lemma~\ref{lemma_Picard}. Together with \eqref{ineq_helper_am}, we obtain
	\begin{equation*}
		\sum_{m=1}^M \norm{a^m}_\ltN^2 
		\overset{\eqref{ineq_helper_am}}{=} 
		\sum\limits_{k = 1}^\infty \sum\limits_{j=1}^{\rk}\frac{1}{\mukj^2}\abs{ \kl{ \spr{y_n,f_k^n}_{Y_n}}_{n=1}^N }^2 
		\overset{\eqref{cond_Picard}}{<} \infty \,.
	\end{equation*}	
Hence, it follows that $a^m \in \ltN$ for each $m \in \Kl{1\,,\dots\,, M}$. Furthermore, since we assumed that the frames $\{e_k^m \}_{k\in\N}$ are exact, it follows from \eqref{eq_frame_equivalences} that $R(F_m^*) = 0$. and thus that $R(F_m) = \ltN$. Hence, we have that $a^m \in R(F_m)$ for each $m \in \Kl{1\,,\dots\,, M}$, which shows that \eqref{condition_Range} holds. As a result, Theorem~\ref{thm_main_II} is applicable, which yields the assertions and thus concludes the proof.
\end{proof}

Table~\ref{table_requirements} summarizes the assumptions used above to show that $\AD y$ is either the unique solution or a minimum-coefficient/minimum-norm (least squares) solution of~\eqref{Ax=y}. 

\begin{table}[ht!]
	\centering
	\begin{scriptsize}
		\begin{tabular}{ | l | c  c  c  c  c  c |}
			\hline
			\multirow{2}{*}{Solution type of $\AD y$} & \multirow{2}{*}{Assumption~\ref{ass_main}}  & $N(\Lk) = \Kl{0}$ & $\{f_k^n\}_{k\in\N}$ & \eqref{condition_Range}; or $\{e_k^m\}_{k\in\N}$  & \multirow{2}{*}{\eqref{cond_range_solvable}} & $\{e_k^m\}_{k\in \N}$ 
			\\
			& & $y \in R(A)$ & tight & exact and \eqref{cond_Picard} & & tight 
			\\ 
			\hline 
			\multirow{2}{*}{unique solution} & \mcX & \mcX & & & & 
			\\
			 & & & & &  &
			\\
			\multirow{2}{*}{min.-coeff. least-squares} & \mcX & & \mcX & \mcX & & 
			\\
			 & & & & & & 
			\\
			\multirow{2}{*}{min.-coeff. solution} & \mcX & & \mcX & \mcX & \mcX &  
			\\
			 & & & & & & 
			\\
			\multirow{2}{*}{min.-norm least-squares} & \mcX & & \mcX  & \mcX &  & \mcX  
			\\
			 & & & & & & 
			\\
			\multirow{2}{*}{min.-norm solution} & \mcX &  & \mcX  & \mcX & \mcX & \mcX 
			\\
			& & & & & & 
			\\
			\hline			
		\end{tabular}
	\end{scriptsize}
	\caption{Summary of requirements guaranteeing $\AD y$ to be a certain type of solution.}
	\label{table_requirements}
\end{table}

\vspace{0pt}
\begin{remark} 
In Theorem~\ref{thm_main_I} and Corollary~\ref{corollary_main_I} we assumed that the frames $\{f_k^n\}_{k\in\N}$ are tight to deduce that $\AD y$ is a least squares solution of \eqref{Ax=y}. However, even if that is not the case, then it follows from the fact that the expansion coefficients \eqref{AD_canonic_expansion_coefficients} of $\AD y$ are minimizers of the functionals \eqref{min_problems} that for any $x = (x_m)_{m=1}^M \in X$ there holds
	\begin{equation*}
	\begin{split}
		& C_1 \norm{A (\AD y) - y}_Y^2 
		\overset{\eqref{eq_residual}}{\leq} 
		\sum_{k=1}^\infty
		\norm{ \Lk \cdot \kl{\spr{(\AD y)_m,e_k^m}_{X_m}}_{m=1}^M - \kl{ \spr{y_n,f_k^n}_{Y_n}}_{n=1}^N }_{\R^N}^2  
		\\
		& \qquad
		\leq 
		\sum_{k=1}^\infty
		\norm{ \Lk \cdot \kl{\spr{x_m,e_k^m}_{X_m}}_{m=1}^M - \kl{ \spr{y_n,f_k^n}_{Y_n}}_{n=1}^N }_{\R^N}^2 
		\overset{\eqref{eq_residual}}{\leq}
		C_2 \norm{Ax - y}_Y^2 \,.
	\end{split}
	\end{equation*}	
In particular, for any least squares solution $x^*$ of \eqref{Ax=y} there holds	
	\begin{equation*}
		\norm{A (\AD y) - y}_Y
		\leq
		\sqrt{C_2/C_1} \norm{A x^* - y}_Y \,,
	\end{equation*} 
which implies that even in the case that the frames $\{f_k^n\}_{k \in \N}$ are not tight, the function $\AD y$ is at most a factor $\sqrt{C_2/C_1}$ away from being a least-squares solution of \eqref{Ax=y}. Furthermore, this shows that if \eqref{Ax=y} is solvable then $\AD y$ is a solution of \eqref{Ax=y} given that either \eqref{cond_range_solvable} holds or that the frames $\{e_k^m\}_{k\in\N}$ are exact.
\end{remark}\vspace{7pt}

\begin{remark}
The above results can be generalized by allowing a more general linear relationship between the coefficients $\spr{A_n x,f_k^n}_{Y_n}$ and $\spr{x_m,e_j^m}_{X_m}$ than the one in \eqref{cond_frames}. In particular, it is possible to allow a linear relationship between coefficients corresponding to multiple different values of $j$ and $k$, as long as there is no ``overlap''. More precisely, one can assume that there exist $(N \cdot K(X,k)) \times (M \cdot K(Y,k))$ matrices $\Lkb$ such that
	\begin{equation}\label{eq_frames_generalized}
		\kl{\kl{\spr{A_n x,f_j^n}_{Y_n}}_{n=1}^N }_{j \in K(Y,k)} = \Lkb \cdot \kl{\kl{\spr{x_m,e_j^m}_{X_m}}_{m=1}^M }_{j \in K(X,k)} \,,
		\qquad 
		\forall \,k \in \N \, \,,
	\end{equation}
where the finite index sets $K(X,k) \subset \N$ and $K(Y,k) \subset \N$ are such that the families $\{K(X,k)\}_{k\in\N}$ and $\{K(Y,k)\}_{y\in\N}$ are pairwise disjoint partitions of $\N$. Hence, each of the coefficients $\spr{A_n x,f_j^n}_{Y_n}$ and $\spr{x_m,e_j^m}_{X_m}$ appears in only one linear relationship, i.e., only for one value of $k$. Condition \eqref{cond_frames} then corresponds to the special case that the matrices $\Lkb$ are block-diagonal, and thus that \eqref{eq_frames_generalized} decouples. With this, analogous to the above results can still be proven, which we later use in Section~\ref{sect_tomo_atmos}.
\end{remark}\vspace{7pt}

\begin{remark}
Frame decompositions can be used to define stable approximations $\xad$ of $\AD y$ in the presence of noisy data $\yd = (\yd_n)_{n=1}^N \in Y$. Considering for example the definition of $\AD y$ given in  \eqref{def_AD}, then in analogy to \eqref{def_xad} one can define the approximation
	\begin{equation}\label{helper_xad}
		\xad = \sum_{k=1}^\infty \sum\limits_{j=1}^{\rk} g_\alpha(\mukj)  \kl{\ukj^H \,\cdot \kl{ \spr{y_n^\delta,f_k^n}_{Y_n}}_{n=1}^N   }\kl{   \kl{  \vkj }_m \, \ekt^m }_{m=1}^M \,,
	\end{equation}
where $g_\alpha : \R \to \R$ is a suitable approximation of the function $s \mapsto 1/s$, such as
	\begin{equation*}
		g_\alpha(s) := \frac{1}{s+\alpha} \,,
		\qquad
		g_{\alpha,n}(s) = \frac{(s+\alpha)^n - \alpha^n}{s(s + \alpha)^n}
		\,,
		\quad
		\text{or}
		\quad
		g_\alpha(s) := 
		\begin{cases}
			1/s \,, & s \geq \alpha \,,
			\\
			0 \,, & s < \alpha \,.
		\end{cases}
	\end{equation*}
In the SVD case \eqref{def_xad}, these choices correspond to Tikhonov and iterated Tikhonov regularization, as well as the truncated SVD/spectral cut-off method, respectively \cite{Engl_Hanke_Neubauer_1996}. Many iterative regularization methods such as Landweber iteration or the Brakhage $\nu$-methods also have a characterization in terms of such a spectral filter function $g_\alpha$.

From \eqref{helper_xad} it follows as in the proof of Lemma~\ref{lemma_Picard} that
	\begin{equation*}
		\norm{\xa - \xad}_X^2 
		\leq
		\frac{1}{B_1}  
		\sum\limits_{k = 1}^\infty \sum\limits_{j=1}^{\rk}g_\alpha\kl{\mukj}^2\abs{\ukj^H \,\cdot \kl{ \spr{y_n - y_n^\delta,f_k^n}_{Y_n}}_{n=1}^N }^2  \,,
	\end{equation*}
which together with the Cauchy-Schwarz inequality and  \eqref{eq_SVD_sigma} implies
	\begin{equation*}
		\norm{\xa - \xad}_X^2 
		\leq
		\frac{1}{B_1}
		\max\limits_{k\in \N}\Kl{\sum\limits_{j=1}^{\rk} g_\alpha(\mukj)^2} 
		\sum\limits_{k = 1}^\infty
		\sum\limits_{n=1}^N \abs{ \spr{y_n - y_n^\delta,f_k^n}_{Y_n} }^2   \,.
	\end{equation*}
Now, since the set $\{f_k^n\}_{k \in \N}$ forms a frame over $Y_n$, it follows from \eqref{eq_frame_fk} that
	\begin{equation*}
		\norm{\xa - \xad}_X^2 
		\leq
		(C_2/B_1)
		\max\limits_{k\in \N}\Kl{\sum\limits_{j=1}^{\rk} g_\alpha(\mukj)^2} 
		\norm{y-y^\delta}_Y^2  \,.
	\end{equation*}
Hence, if the function $g_\alpha$ is chosen such that $g_\alpha(\mukj)$ remains bounded for $k \to \infty$, then the approximations $\xad$ depend continuously on the data $\yd$. In fact, we could retrace all steps of the standard convergence analysis of regularization methods for linear inverse problems \cite{Engl_Hanke_Neubauer_1996} to obtain convergence of $\xad$ to $\AD y$ under standard assumptions on $g_\alpha$, thereby extending classic results from the SVD to the frame decomposition.
\end{remark}\vspace{7pt}

\begin{remark}
For the numerical computation of the functions $\AD y$ one needs to be able to evaluate the dual frame functions $\ekt^m$. In most cases, these functions cannot be computed analytically, but one may approximate them using the iterative approximation formula \eqref{eq_dual_N_approx}. Due to \eqref{eq_dual_approx_error} only a small number of iterations are necessary, if the frame bounds $B_1$ and $B_2$ are close to each other. Note that the functions $\ekt^m$ are independent of the actual right hand side $y$ of equation \eqref{Ax=y} and can thus be computed in advance. This is particularly useful if the problem has to be solved multiple times, since then the computed approximations of $\ekt$ can be stored and re-used. Additionally, it is possible that in some situations a frame decomposition might behave better than the SVD.
\end{remark}

\section{Applicability to specific operator classes}\label{sect_discussion}

We now turn our attention to Assumption~\ref{ass_main}, i.e., to the assumption that there exists a sequence $\{\Lk \}_{k\in\N}$ of complex-valued matrices $\Lk = (\lkmn)_{m,n=1}^{M,N}$ such that
	\begin{equation*}
		\kl{\spr{A_n x,f_k^n}_{Y_n}}_{n=1}^N = \Lk \cdot \kl{\spr{x_m,e_k^m}_{X_m}}_{m=1}^M \,,
		\qquad 
		\forall \,k \in \N \,, \,  x = (x_m)_{m=1}^M \in X \,,
	\end{equation*}
and in particular to the task of finding frames which satisfy this. 

First, note that Assumption~\ref{ass_main} can sometimes be satisfied by choosing frames which are suitably adapted to the ``geometric'' structure of the considered problem. As we shall see on some examples in Section~\ref{sect_numerical_results}, operators composed mainly of shifting and scaling operations can often be decomposed using frames derived from exponential bases. These can also sometimes be used to extend existing decompositions over regular domains to irregular domains. The same is also true for differential operators. In addition, wavelet frames can often be suitable, in particular since one can choose from a wide variety of available wavelets with properties such as regularity, vanishing moments, or compact support. This provides a link to the Wavelet-Vaguelett decomposition mentioned above.

Next, we shall see that Assumption~\ref{ass_main} can also be satisfied if one of the following three situations occurs:
\begin{enumerate}
	\item The singular-value decomposition of $A \, : X \to Y$ is known.
	\item The operator $A \, : X \to Y$ satisfies a stability property of the form
		\begin{equation}\label{cond_helper_1}
			c_1 \norm{x}_X \leq \norm{Ax}_Z \leq c_2 \norm{x}_X \,,
			\qquad
			\forall \, x \in X \,,
		\end{equation}	
	with constants $c_1,c_2 > 0$ and some Hilbert space $Z$ being a (dense) subspace of~$Y$. 
	\item The operator $A$ is continuously invertible (i.e.\ it satisfies \eqref{cond_helper_1} with $Z=Y$).
\end{enumerate}

We shall now discuss each of those situations in turn, devoting particular attention to the second (and third) situation in Sections~\ref{sect_smooth_I} and \ref{sect_smooth_II} below. Even though it allowed us to derive more general results in the previous section, for the subsequent considerations we do not need to make explicit use of the general structure \eqref{def_X_Y} of the spaces $X$ and $Y$. Thus, we now consider 
	\begin{equation*}
	\begin{split}
		A \, : \, X \to \,	Y  \,,
		\qquad		
		x  \, \mapsto \, Ax  \,,
	\end{split}
	\end{equation*}
as a special case of \eqref{def_X_Y} with $M=N=1$. In order to keep the notation simple, in what follows we drop all sub- and superscripts related to $M$ and $N$. For example, instead of writing $e_k^1$ and $f_k^1$ for the frame functions, we now simply write $e_k$ and $f_k$, and so on. With this, condition \eqref{cond_frames} in Assumption~\ref{ass_main} reads
	\begin{equation}\label{cond_frames_1D}
		\spr{A x,f_k}_{Y} = \lk \spr{x,e_k}_{X} \,,
		\qquad 
		\forall \,k \in \N \,, \,  x  \in X \,,
	\end{equation}
where $\lk$ now is a sequence of complex numbers instead of matrices. Furthermore, the expressions \eqref{eq_A_dec} and \eqref{def_AD} for the operators $A$ and $\AD$ now respectively read
	\begin{equation}\label{eq_dec_A_AD_1D}
		A x = \sum_{k=1}^\infty \lk \spr{x,e_k} \fkt \,,
		\qquad
		\text{and}
		\qquad
		\AD y = \sum\limits_{\underset{\lk\neq 0}{k=1}}^\infty \frac{\spr{y,f_k}}{\lk}  \ekt \,,
	\end{equation}
and the Picard condition \eqref{cond_Picard} turns into
	\begin{equation}\label{cond_Picard_1D}
		\sum\limits_{\underset{\lk \neq 0}{k=1}}^\infty \frac{\abs{\spr{y,f_k}_{Y} }^2 }{\abs{\lk}^2}
		< 
		\infty \,.
	\end{equation}
Note that if $X$ and $Y$ have a structure of the form \eqref{def_X_Y}, then condition \eqref{cond_frames} is a more general linear relationship between the operator and the frames than condition \eqref{cond_frames_1D}. This is important if none of the three situations introduced above apply, in which case a frame decomposition can still be obtained under condition \eqref{cond_frames} as in Section~\ref{sect_frame_dec}.

\subsection{The SVD as a Frame Decomposition}

Comparing \eqref{eq_dec_A_AD_1D}, \eqref{cond_Picard_1D} with \eqref{def_A_dec}, \eqref{def_Ad}, \eqref{Picard}, the similarities between the frame decomposition and the singular-value decomposition become apparent. In particular, since the singular functions form orthonormal bases of $\ol{R(A)}$ and $N(A)^\perp$, we obtain

\begin{lemma}\label{lem_svd_frames}
Let $(\sk, v_k, u_k)_{k=1}^\infty$ be the singular system of $A$ and let $\{m_k\}_{k\in\N}$ and $\{n_k\}_{k\in\N}$ be orthonormal bases of $N(A)$ and $N(A^*)$, respectively. Then the sets
	\begin{equation}\label{def_frame_svd}
		\{e_k\}_{k\in\N} := \{v_k \}_{k\in \N} \cup \{m_k\}_{k\in\N}
		\qquad
		\text{and}
		\qquad
		\{f_k\}_{k\in\N} := \{u_k \}_{k\in \N} \cup \{n_k\}_{k\in\N}
	\end{equation}
form tight frames with frame bound $1$ over the spaces $X$ and $Y$, respectively. Together with $\{\lk\}_{k\in\N} := \{ \sk \}_{k \in \N} \cup \{ 0 \}_{k \in \N}$, they also satisfy condition \eqref{cond_frames_1D}.
\end{lemma}
\begin{proof}
This is a direct consequence of the properties of the SVD.
\end{proof}

Due to the above result, it follows that it is possible to find a frame decomposition for any bounded (compact) linear operator $A$, at least in theory. Of course this result is not very practical, since it again involves the SVD, which we initially set out to avoid. However, since by Lemma~\ref{lem_svd_frames} the frames $\{e_k\}_{k\in\N}$ and $\{f_k\}_{k\in\N}$ are tight with frame bound $1$ and thus $\ekt = e_k$ and $\fkt = f_k$, it follows that the results on the frame decomposition derived above are a direct generalization of the classic results on the SVD (compare for example with Corollary~\ref{corollary_main_I}).

\subsection{Stability Property - Part I}\label{sect_smooth_I}

Let us now turn our attention to the second situation, i.e., to the case that $A$ satisfies a stability property of the form \eqref{cond_helper_1}. For this, we start by considering condition \eqref{cond_frames_1D}, which is clearly equivalent to
	\begin{equation}\label{cond_frames_connected}
		\lkc \, e_k =  \, A^*f_k \,,
		\qquad
		\forall \,k \in \N \,,
	\end{equation}
with $\lkc$ denoting the complex conjugate of $\lk$, and is reminiscent of the connection \eqref{def_singular_system} between the singular values and functions. This suggests the following strategy for finding frames that fulfil \eqref{cond_frames_connected}: Starting with some frame $\{f_k\}_{k\in\N}$, one simply defines
	\begin{equation}\label{def_ek_frame}
		e_k := \frac{1}{\lkc} A^* f_k \,, 
	\end{equation}
for some sequence of coefficients $\{\lk\}_{k \in \N}$. If one can choose the $\lk$ in such a way that the resulting set $\{e_k\}_{k\in\N}$ forms a frame over $X$, then all assumptions in Theorem~\ref{thm_main_I} are satisfied and thus the results on the frame decomposition derived above hold. As we are going see now, this is possible if the operator $A$ has a specific stability property, which we now introduce in the following

\begin{assumption}\label{ass_main_II}
The operator $A : X \to Y$ satisfies the condition
	\begin{equation}\label{cond_A_stability}
		c_1 \norm{x}_X \leq \norm{Ax}_Z \leq c_2 \norm{x}_X \,,
		\qquad
		\forall \, x \in X \,,
	\end{equation}	
for some constants $c_1,c_2 > 0$, where $Z \subseteq Y$ is a Hilbert space. Furthermore, there exists a sequence of coefficients $0 \neq \ak \in \R$ and some constants $a_1,a_2 > 0$ such that
	\begin{equation}\label{cond_norm_Z}
		a_1 \norm{y}_Z^2 
		\leq \sum_{k=1}^\infty \ak^2  \abs{\spr{y,f_k}_Y}^2 
		\leq a_2 \norm{y}_Z^2
		 \,,
		\qquad
		\forall \, y \in Y \,,
	\end{equation} 
where as before the functions $f_k$ are such that the set $\{f_k\}_{k\in\N}$ forms a frame over $Y$.
\end{assumption}

Condition \eqref{cond_A_stability} is satisfied for many operators of practical relevance, the most prominent example perhaps being the Radon transform (compare with Section~\ref{sect_numerical_results} below). It implies that $A$ is injective, and that as an operator from $X$ to $Z$ it is continuously invertible. However, in the presence of noise the right-hand side $y$ in equation  \eqref{Ax=y} typically only belongs to $Y$ but not to $Z$. Condition \eqref{cond_norm_Z} is for example satisfied if $Y$ and $Z$ are Sobolev spaces and $\{f_k\}_{k\in\N}$ is a suitably chosen wavelet or exponential frame/basis (again see Section~\ref{sect_numerical_results} below). We now proceed to derive the following 
\begin{proposition}\label{prop_stability}
Let $A: X \to Y$ be a bounded linear operator and let Assumption~\ref{ass_main_II} hold. Furthermore, let the functions $e_k$ be defined by \eqref{def_ek_frame}, where the parameters $\lk \in \C$ are chosen such that
	\begin{equation}\label{cond_sigma_k_bounds}
		0 <  b_1 \leq \ak \abs{\lk} \leq  \, b_2 < \infty \,, 
		\qquad
		\forall \, k \in \N \,,
	\end{equation} 
for some constants $b_1, b_2 > 0$. Then the set $\{e_k\}_{k\in \N}$ forms a frame over $X$ with frame bounds $B_1 = a_1 (c_1/b_2)^2 $ and $B_2 = a_2 (c_2/b_1)^2 $.
\end{proposition}
\begin{proof}
Let $x \in X$ be arbitrary but fixed. Due to \eqref{cond_A_stability}  and \eqref{cond_norm_Z} there holds
	\begin{equation}\label{eq_proof_helper_1}
	\begin{split}
		a_1 c_1^2 \norm{x}_X^2
		\overset{\eqref{cond_A_stability}}{\leq}  
		a_1 \norm{Ax}_Z^2
		\overset{\eqref{cond_norm_Z}}{\leq} 
		\sum_{k=1}^\infty \ak^{2} \abs{\spr{Ax,f_k}_Y}^2
		\overset{\eqref{cond_norm_Z}}{\leq}
		a_2  \norm{Ax}_Z^2 
		\overset{\eqref{cond_A_stability}}{\leq} 
		a_2 c_2^2 \norm{x}_X^2 \,.  		
	\end{split}
	\end{equation}	
Furthermore, it follows from \eqref{cond_sigma_k_bounds} that
	\begin{equation}\label{eq_proof_helper_2}
		b_2^{-2}\sum_{k=1}^\infty \ak^{2} \abs{\spr{Ax,f_k}_Y}^2 
		\leq
		\sum_{k=1}^\infty \abs{\lk}^{-2} \abs{\spr{Ax,f_k}_Y}^2 
		\leq
		b_1^{-2}\sum_{k=1}^\infty \ak^{2} \abs{\spr{Ax,f_k}_Y}^2 \,, 
	\end{equation}	
and thus together with \eqref{eq_proof_helper_1} we obtain
	\begin{equation}\label{eq_proof_helper_3}
		a_1 (c_1/b_2)^2 \norm{x}_X^2 
		\leq 
		\sum_{k=1}^\infty \abs{\lk}^{-2} \abs{\spr{Ax,f_k}_Y}^2 
		\leq
		a_2 c_2^2 \norm{x}_X^2 \,. 
	\end{equation}				
Now since by the definition of $e_k$ there holds
	\begin{equation}\label{eq_proof_helper_4}
		\sum_{k=1}^\infty \abs{\spr{x,e_k}_X}^2
		=
		\sum_{k=1}^\infty \abs{\spr{x,\lk^{-1} A^*f_k}_X}^2
		=
		\sum_{k=1}^\infty \abs{\lk}^{-2} \abs{\spr{Ax,f_k}_Y}^2 \,,
	\end{equation}
it follows from \eqref{eq_proof_helper_3} that
	\begin{equation*}
		a_1 (c_1/b_2)^2 \norm{x}_X^2
		\leq
		\sum_{k=1}^\infty \abs{\spr{x,e_k}_X}^2
		\leq
		a_2 (c_2/b_1)^2\norm{x}_X^2
		\,,
	\end{equation*}
which shows that $\{e_k\}_{k\in \N}$ forms a frame over $X$ and thus yields the assertion.
\end{proof}

Using the above proposition, we can now derive the third main result of this paper: 
\begin{theorem}\label{thm_main_III}
Let $A: X \to Y$ be a bounded linear operator and let Assumption~\ref{ass_main_II} hold. Furthermore, let the functions $e_k$ be defined by \eqref{def_ek_frame}, where the parameters $\lk \in \C$ are chosen such that \eqref{cond_sigma_k_bounds} holds. Then for any $y \in Z$ the function $\AD y$ as defined in \eqref{eq_dec_A_AD_1D} is the unique solution of \eqref{Ax=y}. Among all possible decompositions of $\AD y$ in terms of the dual frame functions $\ekt$, decomposition~\eqref{eq_dec_A_AD_1D} is the most economical one in the sense of Proposition~\ref{prop_frame_min}.
\end{theorem}
\begin{proof}
Due to Proposition~\ref{prop_stability} the set $\{e_k\}_{k\in\N}$ forms a frame over $X$. Furthermore, it follows from \eqref{cond_sigma_k_bounds} that $\lk \neq 0$ for all $k \in \N$. Moreover, since due to \eqref{cond_A_stability} there holds $R(A) = Z$, it follows that \eqref{Ax=y} is (uniquely) solvable for any $y \in Z$. Hence, all conditions of Theorem~\ref{thm_main_I} are satisfied, which yields the assertions
\end{proof}

\begin{remark}
Due to \eqref{cond_norm_Z} and \eqref{cond_sigma_k_bounds} it follows that
	\begin{equation*}
		a_1 b_2^{-2}\norm{y}_Z^2
		\leq
		\sum_{k=1}^\infty \frac{\abs{\spr{y,f_k}_Y}^2}{\abs{\lk}^2}
		\leq 
		a_2 b_1^{-2}\norm{y}_Z^2  \,.
	\end{equation*}
This implies that the Picard condition \eqref{cond_Picard_1D} holds if and only if $\norm{y}_Z < \infty$, i.e.\ if $y \in Z$. Hence, it also follows that $\AD y$ is well-defined for any $y \in Z$. This should be compared to the definition space of the Moore-Penrose inverse $\Ad y$, for which there holds \cite{Engl_Hanke_Neubauer_1996}
	\begin{equation*}
		D(\Ad) = R(A) + R(A)^\perp \,.
	\end{equation*}
Note first that from \eqref{cond_A_stability} it follows that $R(A) = Z$. Now since $Z$ is typically a dense subspace of $Y$, it follows that $Z^\perp = 0$, and thus we get that $D(\Ad) = Z \subseteq D(\AD)$. Hence, the definition space of $\AD$ is at least as large as the definition space of $\Ad$.
\end{remark}\vspace{7pt}

A number of simplifications of the above theory are possible if the operator $A$ is continuously invertible. This is because in that case condition \eqref{cond_A_stability} is satisfied for the choice $Z = Y$. Consequently, also condition \eqref{cond_norm_Z} is satisfied for any frame $\{f_k\}_{k\in\N}$ over $Y$ together with $\ak = 1$. Hence, we can obtain the following

\begin{theorem}\label{thm_main_IIII}
Let $A: X \to Y$ be a bounded and continuously invertible linear operator and let $\{f_k\}_{k\in \N}$ form a frame over $Y$. Furthermore, let the functions $e_k$ be defined by \eqref{def_ek_frame}, where the parameters $\lambda_k \in \C$ are such that for some constants $b_1, b_2 > 0$ 
	\begin{equation}\label{eq_bounds_sk_Ainv}
		0 <  b_1 \leq \abs{\lk} \leq  \, b_2 < \infty \,, 
		\qquad
		\forall \, k \in \N \,.
	\end{equation}  
Then for any $y \in Y$ the function $\AD y$ as defined in \eqref{eq_dec_A_AD_1D} is the unique solution of \eqref{Ax=y}. Among all possible decompositions of $\AD y$ in terms of the dual frame functions $\ekt$, decomposition~\eqref{eq_dec_A_AD_1D} is the most economical one in the sense of Proposition~\ref{prop_frame_min}.
\end{theorem}
\begin{proof}
Since $A$ is a bounded and continuously invertible operator it satisfies \eqref{cond_A_stability} for $Z = Y$. Furthermore, since $\{f_k\}_{k\in\N}$ forms a frame over $Y$, it follows that also \eqref{cond_norm_Z} is satisfied for $Z = Y$ and $\ak = 1$. Moreover, due to \eqref{eq_bounds_sk_Ainv} also condition \eqref{cond_sigma_k_bounds} holds. Hence, Theorem~\ref{thm_main_III} is applicable, which yields the assertion.
\end{proof}

\subsection{Stability Property - Part II}\label{sect_smooth_II}

Next, we turn our attention to a slightly different way of deriving a frame decomposition for operators satisfying the stability property \eqref{cond_A_stability}. This approach can be used even if no frame $\{f_k\}_{k\in\N}$ satisfying \eqref{cond_norm_Z} is known or (numerically) feasible. All one needs is that the functions $f_k$ are elements of the subspace $Z$, albeit at the cost of a (numerically) more involved determination of a suitable frame $\{e_k\}_{k\in\N}$. This leads us to the following

\begin{assumption}\label{ass_main_III}
The operator $A : X \to Y$ satisfies condition \eqref{cond_A_stability}, i.e,
	\begin{equation*}
		c_1 \norm{x}_X \leq \norm{Ax}_Z \leq c_2 \norm{x}_X \,,
		\qquad
		\forall \, x \in X \,, 
	\end{equation*}	
for some constants $c_1,c_2 > 0$, where the Hilbert space $Z \subseteq Y$ is a dense subspace of $Y$, for which there holds
	\begin{equation}\label{cond_Y_Z}
		\norm{y}_Y \leq \norm{y}_Z  \,, 
		\qquad \forall \, y \in Z \,.
	\end{equation} 
Furthermore, the functions $f_k$ are such that the set $\{f_k\}_{k\in\N}$ forms a frame over $Y$ with frame bounds $C_1,C_2>0$. Additionally, these $f_k$ are elements of $Z$, i.e.,  $\norm{f_k}_Z < \infty$.
\end{assumption}

It is known (see e.g.~\cite{Engl_Hanke_Neubauer_1996}) that if \eqref{cond_Y_Z} holds then there exists a densely defined, unbounded, selfadjoint, strictly positive operator $L$ with $D(L) = Z$ such that 
	\begin{equation}\label{eq_L_Y_Z}
		\norm{L y}_Y = \norm{y}_Z \,, 
		\qquad \forall \, y \in Z \,.
	\end{equation}
This operator $L$ is uniquely determined by $L = (E E^*)^{-1/2}$, where $E:Z \to Y$ denotes the embedding operator. With this, we can proceed to derive the following 

\begin{lemma}\label{lem_ek_L_frame}
Let $A : X \to Y$ be a bounded linear operator and let Assumption~\ref{ass_main_III} hold. Then the set $\{e_k\}_{k\in \N}$, where the functions $e_k$ are defined as
	\begin{equation}\label{def_ek_L}
		e_k := A^*L f_k \,,
	\end{equation}
form a frame over $X$ with frame bounds $B_1= c_1^2\, C_1$ and $B_2 = c_2^2 \,C_2$, where $C_1$ and $C_2$ are the frame bounds of $\{f_k\}_{k\in\N}$, and $c_1$ and $c_2$ are as in Assumption~\ref{ass_main_III}.
\end{lemma}
\begin{proof}
First of all, due to Assumption~\ref{ass_main_III} it follows that
	\begin{equation*}
		\norm{e_k}_X = \norm{A^*Lf_k}_X \leq \norm{A^*} \norm{L f_k}_Y  = \norm{A} \norm{f_k}_Z < \infty \,, 
	\end{equation*}
and thus the functions $e_k$ are well-defined. Let now $x \in X$ be arbitrary but fixed and note that due to \eqref{cond_A_stability} there holds $Ax \in Z = D(L)$. First, since $L$ is selfadjoint we get
	\begin{equation}\label{eq_helper_3}
		\sum_{k=1}^{\infty} \abs{\spr{x,e_k}_X}^2 
		= \sum_{k=1}^{\infty} \abs{\spr{x,A^*Lf_k}_X}^2 
		= \sum_{k=1}^{\infty} \abs{\spr{L A x, f_k}_Y}^2 \,.
	\end{equation}
Furthermore, since the set $\{f_k\}_{k\in\N}$ forms a frame over $Y$ with frame bounds $C_1,C_2$, it follows with \eqref{eq_L_Y_Z} that
	\begin{equation*}
	\begin{split}
		C_1 \norm{A x}_Z^2
		\overset{\eqref{eq_L_Y_Z}}{=}
		C_1 \norm{L A x}_Y^2
		\leq
		\sum_{k=1}^{\infty} \abs{\spr{L A x, f_k}_Y}^2 
		\leq C_2 \norm{L A x}_Y^2
		\overset{\eqref{eq_L_Y_Z}}{=}
		C_2 \norm{Ax}_Z^2 \,.
	\end{split}
	\end{equation*}
which combined with \eqref{eq_helper_3} yields
	\begin{equation}\label{eq_helper_2}
	\begin{split}
		C_1 \norm{A x}_Z^2
		\leq
		\sum_{k=1}^{\infty} \abs{\spr{x,e_k}_X}^2
		\leq
		C_2 \norm{Ax}_Z^2 \,.
	\end{split}
	\end{equation}
Hence, together with \eqref{cond_A_stability} we obtain
	\begin{equation*}
		c_1^2 \, C_1 \norm{x}_X^2 
		\overset{\eqref{cond_A_stability}}{\leq}
		C_1 \norm{A x}_Z^2 
		\overset{\eqref{eq_helper_2}}{\leq} 
		\sum_{k=1}^{\infty} \abs{\spr{x,e_k}_X}^2 
		\overset{\eqref{eq_helper_2}}{\leq} 
		C_2 \norm{A x}_Z^2 
		\overset{\eqref{cond_A_stability}}{\leq} 
		c_2^2 \, C_2 \norm{x}_X^2 \,,
	\end{equation*}	
which yields the assertion.
\end{proof}

Instead of the original problem  \eqref{Ax=y} we now consider the ``preconditioned'' equation
	\begin{equation}\label{LAx=Ly}
		L A x = L y \,.
	\end{equation}
If condition \eqref{cond_A_stability} holds, then the concatenated operator $L A$ is bounded, linear, and continuously invertible from $X$ to $Y$. Hence, both problems \eqref{LAx=Ly} and \eqref{Ax=y} are uniquely solvable if and only if $Ly \in Y$, which due to \eqref{eq_L_Y_Z} is equivalent to $y \in Z$. For this case, we want to derive an expression of the solution in terms of the frames $\{e_k\}_{k\in\N}$ and $\{f_k\}_{k\in\N}$. We start by giving a frame decomposition of the operator $LA$ in the following

\begin{lemma}\label{lem_dec_LA}
Let $A : X \to Y$ be a bounded linear operator and let Assumption~\ref{ass_main_III} hold. Furthermore, let the functions $e_k$ be defined as in \eqref{def_ek_L}. Then for all $x \in X$ there holds
	\begin{equation}\label{eq_LAx}
		\spr{LAx,f_k}_Y = \spr{x,e_k}_X \,,
	\end{equation}
and consequently
	\begin{equation}\label{eq_dec_LA}
		L A x = \sum_{k=1}^\infty \spr{x,e_k}_X \fkt \,. 
	\end{equation}	
\end{lemma}
\begin{proof}
Due to the definition \eqref{def_ek_L} of the functions $e_k$ there holds
	\begin{equation*}
		\spr{x,e_k}_X  = \spr{x,A^* Lf_k}_X = \spr{LAx,f_k}_Y  \,.
	\end{equation*}
Hence, since the set $\{f_k\}_{k\in\N}$ forms a frame, equation \eqref{eq_dec_LA} now follows from \eqref{eq_frame_rec}.
\end{proof}

Note that by applying the well-defined inverse operator $L^{-1} = (EE^*)^{1/2}$ to \eqref{eq_dec_LA} we also obtain an expression for the operator $A$, namely
	\begin{equation*}
		A x = L^{-1} \kl{\sum_{k=1}^\infty \spr{x,e_k}_X  \fkt }\,. 
	\end{equation*}
Next, we proceed to make the following
\begin{definition}
The operator $\ADb : Y \to X$ is defined by 
	\begin{equation}\label{def_ADt}
		\ADb y = \sum_{k=1}^\infty \spr{Ly ,f_k}_Y \ekt \,.
	\end{equation}
\end{definition}
For this operator $\ADb$, we can derive the following well-definedness result:
\begin{lemma}\label{lem_ADb_bounded}
Let $A:X \to Y$ be a bounded linear operator satisfying Assumption~\ref{ass_main_III}. Furthermore, let $e_k$ be defined as \eqref{def_ek_L}. Then for any $y \in Z$ the function $\ADb y$ given in \eqref{def_ADt} is a well-defined element of $X$.
\end{lemma}
\begin{proof}
It follows from Lemma~\ref{lem_ek_L_frame} that the set $\{e_k\}_{k\in\N}$ forms a frame over $X$ with some frame bounds $B_1$ and $B_2$. Since the dual-frame $\{\ekt\}_{k\in\N}$ then forms a frame with bounds $B_2^{-1}$ and $B_1^{-1}$, it follows from \eqref{eq_bound_F_Fadj} that
	\begin{equation*}
		\norm{\ADb y }_X =  \norm{\sum_{k=1}^\infty \spr{Ly ,f_k}_Y \ekt}_X  
		= \norm{\Ft^*\kl{\spr{Ly ,f_k}_Y}_{k=1}^\infty }_X 
		\leq \frac{1}{\sqrt{B_1}} \norm{ \kl{\spr{Ly ,f_k}_Y}_{k=1}^\infty }_{\lt(\N)} \,.
	\end{equation*}
Now since the set $\{f_k\}_{k\in\N}$ forms a frame over $Y$ with some frame bounds $C_1$ and $C_2$, it follows together with \eqref{eq_L_Y_Z} that
	\begin{equation*}
		\norm{ \kl{\spr{Ly ,f_k}_Y}_{k=1}^\infty }_{\lt(\N)}^2
		=  \sum_{k=1}^\infty \abs{\spr{Ly ,f_k}_Y}^2
		\leq C_2 \norm{Ly}_Y^2 = C_2 \norm{y}_Z^2 \,.
	\end{equation*}
Combining the above we get that if $y \in Z$ then $\ADb y \in X$, which yields the assertion. 
\end{proof}

We can now proceed to derive the following 
\begin{theorem}\label{thm_main_IIIII}
Let $A:X \to Y$ be a bounded linear operator satisfying Assumption~\ref{ass_main_III} hold, and let the functions $e_k$ be defined as in \eqref{def_ek_L}. Then for any $y \in Z$ the function $\ADb y$ as defined in \eqref{def_ADt} is the unique solution of the linear operator equation \eqref{Ax=y}.
\end{theorem}
\begin{proof}
For any $y \in Z$ it follows from \eqref{cond_A_stability} that there exists a unique solution $x^* \in X$ of equation \eqref{LAx=Ly}. Since by Assumption~\ref{ass_main_III} the set $\{f_k\}_{k\in\N}$ forms a frame with frame bounds $C_1,C_2 > 0$, it follows that
	\begin{equation*}
		\sum_{k=1}^\infty \abs{\spr{LA x^* - Ly , f_k }_Y }^2
		\leq 
		C_2 \norm{LA x^* - L y}_{Y}^2 
		= 0 \,,
	\end{equation*}
and thus for each $k\in\N$ there holds
	\begin{equation*}
		\spr{x^*, e_k}_X \overset{\eqref{eq_LAx}}{=} \spr{LAx^*,e_k}_X = \spr{Ly,f_k}_Y \,.
	\end{equation*}
Since by Lemma~\ref{lem_ek_L_frame} the set $\{e_k\}_{k\in\N}$ forms a frame over $X$, it follows from \eqref{eq_frame_rec} that 
	\begin{equation*}
		x^* = \sum_{k=1}^\infty \spr{x^*,e_k}_X \ekt = \sum_{k=1}^\infty \spr{Ly,f_k}_Y \ekt = \ADb y \,.
	\end{equation*}
Hence, the function $\ADb y$ is the unique solution of \eqref{LAx=Ly}. Applying the operator $L^{-1}$ to this equation we see that $\ADb y$ also solves \eqref{Ax=y}. Together with the fact that due to \eqref{cond_A_stability} the nullspace of $A$ is trivial, this yields the assertion.
\end{proof}

\begin{remark}
It can be seen from the proof of Lemma~\ref{lem_ADb_bounded} that
	\begin{equation*}
		\norm{\ADb y}_X \leq \sqrt{C_2/B_1}  \norm{L y}_Y \,.
	\end{equation*}
However, noisy data $\yd$ usually do not belong to the space $Z$, and thus $\ADb \yd$ is not well-defined. Hence, in order to obtain a stable approximation of $\ADb y$ in this case one can, e.g., consider a family $U_\alpha : Y \to Z$ of bounded linear operators and define
	\begin{equation*}
		\xad := \ADb\, U_\alpha \yd \,.
	\end{equation*}
One possible choice for example is $U_\alpha := U := \, L^{-1} = (EE^*)^{1/2}$. Note that since $L$ is often some sort of differential operator (for example if $Y$ and $Z$ are Sobolev spaces), the introduction of this operator $U$ basically amounts to a smoothing of the data $\yd$; compare for example with \cite{Ramlau_Teschke_2004_1,Ramlau_Teschke_2004_2, Klann_Ramlau_2008}. 
\end{remark}\vspace{7pt}

\section{Applications in Tomography} \label{sect_numerical_results}

In this section, we consider the application of our frame decomposition results to some tomographic imaging problems. More precisely, we first consider computerized tomography based on the Radon transform, and then move on to atmospheric tomography.

\subsection{Application to Computerized Tomography}

Many problems of practical importance, for example in industry or in medicine, are based on the well-known \emph{Radon transform} \cite{Natterer_2001,Louis_1989}, which in 2D is given by
	\begin{equation*}\label{Radon}
		(R x) (s,\omega) := \int_\R x(s\omega + t \omega^\perp) \, dt \,.
	\end{equation*}
Together with the parametrisation $\omega = \omega(\vphi) = (\cos(\vphi),\sin(\vphi))^T$ we obtain the operator
	\begin{equation}\label{Radon_A}
	\begin{split}
		(Ax)(s,\vphi) := (Rx)(s,w(\vphi)) 
		= \int_\R x(s\omega(\vphi) + t \omega(\vphi)^\perp) \, dt \,,
	\end{split}	
	\end{equation}
which is the version of the Radon transform commonly used for computational purposes. For the subsequent considerations we first need to recall a number of definitions and results from \cite{Natterer_2001,Louis_1989}, starting with the definition of the Sobolev spaces
	\begin{equation*}
	\begin{split}
		\HaRN &:= \Kl{ y \in \LtRN \, \vert \, \norm{y}_\HaRN < \infty  } \,,
		\\
		\norm{y}_\HaRN &:= \kl{ \int_{\R^N} (1+ \abs{\xi}^2)^\alpha \abs{\hat{y}(\xi)}^2 \, d\xi   }^{1/2} \,.
	\end{split}
	\end{equation*}
Furthermore, for any open subset $\Omega \subset \R^N$ we also define the Sobolev spaces
	\begin{equation*}
		\HazO := \Kl{y \in \HaRN \, \vert \, \supp{f} \subseteq \bar{\Omega}} \,,
	\end{equation*}	
which are equipped with the same norms as $\HaRN$. Next, we introduce the domains $\OD :=  \{x \in \R^2 \, \vert \, \abs{x} \leq 1\}$ and $\OS := \R \times [0,2\pi)$, as well as the Sobolev spaces	
	\begin{equation}\label{def_HaOS}
	\begin{split}
		\HaOS &:= \Kl{y \in \LtOS \, \vert \, \norm{y}_\HaOS < \infty } \,,
		\\
		\norm{y}_\HaOS^2 &:= \int_{0}^{2\pi} \norm{y(\cdot,\vphi)}_\HaR^2 \, d\vphi \,.  
	\end{split}
	\end{equation} 
It has been shown in \cite{Natterer_2001} that for each $\alpha \in \R$ there exist positive constants $c(\alpha)$, $C(\alpha)$ such that 
	\begin{equation}\label{eq_Radon_stability}
		c(\alpha) \norm{x}_{\HazOD} \leq \norm{A x}_{H^{\alpha + 1/2}(\OS)  } \leq C(\alpha) \norm{x}_{\HazOD} \,.
	\end{equation}
This important result can be used to derive the following
\begin{theorem}\label{thm_main_Radon}
Let the Radon transform $A : \HazOD \to \HbOS$ be defined as in \eqref{Radon_A} for some $0 \leq \alpha,\beta \in \R$ satisfying $\beta \leq \alpha + 1/2$. Furthermore, let the functions $f_k$ be such that the set $\{f_k \}_{k\in\N}$ forms a frame over $\HbOS$. Additionally, assume that there exists a sequence of coefficients $0 < \ak \in \R$ such that the norm equivalence
	\begin{equation}\label{cond_norm_HohOS}
		\norm{y}_\HapOS^2 \simeq
		\sum_{k=1}^\infty \ak^2  \abs{\spr{y,f_k}_\HbOS}^2 
		\,,
	\end{equation} 
holds and define $e_k := \ak A^* f_k $. Then the set $\{e_k\}_{k\in\Z}$ forms a frame over $\HazOD$ and
	\begin{equation*}
		A x = \sum\limits_{k=1}^\infty \ak^{-1} \spr{x, e_k}_\HazOD \fkt \,.
	\end{equation*}
Furthermore, for any $y \in \HapOS$ the unique solution of $Ax=y$ is given by 
	\begin{equation}\label{dec_ADy_Radon_general}
		\AD y = \sum\limits_{j,k\in\Z} \ak \spr{y,f_k}_\HbOS \ekt \,.
	\end{equation}
Among all possible decompositions of $\AD y$ in terms of the dual frame functions $\ekt$, the decomposition \eqref{dec_ADy_Radon_general} is the most economical one in the sense of Proposition~\ref{prop_frame_min}.
\end{theorem}
\begin{proof}
First note that since $\beta \leq \alpha + 1/2$, it follows from \eqref{eq_Radon_stability} that $A$ is a well-defined bounded linear operator. Furthermore, it also follows from \eqref{eq_Radon_stability} that \eqref{cond_A_stability} is satisfied with $X = \HazOD$, $Y = \HbOS$, and $Z = \HapOS$. Moreover, due to the norm equivalence \eqref{cond_norm_HohOS} it follows that also \eqref{cond_norm_Z} holds.  Additionally, by definition the functions $e_k$ satisfy \eqref{def_ek_frame} with $\lk = \ak^{-1}$ and thus also \eqref{cond_sigma_k_bounds} also. Hence, Theorem~\ref{thm_main_III} is applicable, which yields the assertion.
\end{proof}
As a consequence of the above result we obtain the following
\begin{theorem}\label{thm_Radon_specific}
Let $0 \leq \alpha \in \R$ and let the Radon transform $A : \HazOD \to \LtOS$ be defined as in \eqref{Radon_A}. Furthermore, let $\{\psi_{j,k}\}_{j,k\in\Z}$ be an orthonormal wavelet basis of $\LtR$, let $\{w_l\}_{l\in\N}$ be an orthonormal basis of $\LtT$, and define the functions 
	\begin{equation*}
		\fjkl(s,\vphi) := \psi_{j,k}(s) w_l(\vphi) \,,
		\qquad
		\text{and}
		\qquad
		\ejkl := \kl{1 + 2^{-2j \alpha}} A^*\fjkl \,.
	\end{equation*}
Then the set $\{\ejkl\}_{j,k\in \Z\,, l \in \N}$ forms a frame over $\HazOD$ and 
	\begin{equation*}
		A x = \sum\limits_{j,k \in \Z}\sum\limits_{l=1}^\infty \kl{1 + 2^{-2j \alpha}}^{-1} \spr{x, \ejkl}_\HazOD \fjklt 
		=
		\sum\limits_{j,k \in \Z}\sum\limits_{l=1}^\infty \spr{x,A^*\fjkl} \fjkl \,.
	\end{equation*}
Furthermore, for any $y \in \HapOS$ the unique solution of $Ax=y$ is given by 
	\begin{equation}\label{dec_ADy_Radon_specific}
		\AD y = \sum\limits_{j,k \in \Z} \sum\limits_{l=1}^\infty \kl{1 + 2^{-2j (\alpha+1/2)}}\spr{y,\fjkl}_\LtOS \ejklt \,.
	\end{equation}
Among all possible decompositions of $\AD y$ in terms of the dual frame functions $\ekt$, the decomposition \eqref{dec_ADy_Radon_specific} is the most economical one in the sense of Proposition~\ref{prop_frame_min}.
\end{theorem}
\begin{proof}
Since $\{\psi_{j,k}\}_{j,k\in\Z}$ is an orthonormal wavelet basis of $\LtR$, the following norm equivalence holds for any $s \in \R$ (see e.g~\cite{Daubechies_1992}):
	\begin{equation*}\label{eq_norm_equiv_wavelets}
		\norm{u}_{H^s(\R)}^2 \simeq \sum\limits_{j,k \in \Z}  (1+ 2^{-2 j s}) \abs{\spr{u,\psi_{j,k}}_{\Lt(\R)}}^2  \,,
	\end{equation*}	
Hence, together with \eqref{def_HaOS} we obtain the norm equivalence	
	\begin{equation*}
		\norm{y}_\HapOS^2 
		\simeq
		\int_{0}^{2\pi} \sum\limits_{j,k \in \Z}  \kl{1+ 2^{-2 j (\alpha+1/2)}} \abs{\spr{y(\cdot,\vphi),\psi_{j,k}}_{\Lt(\R)}}^2 \, d\vphi \,.
	\end{equation*} 
Now, since by its definition the set $\{\fjkl\}_{j,k\in\Z\,, l \in \N}$ forms an orthonormal basis over $\LtOS$, and thus any function $y \in \LtOS$ can be written in the form
	\begin{equation*}
		y = \sum\limits_{j,k \in \Z}\sum\limits_{l=1}^\infty \spr{y,\fjkl}_\LtOD \fjkl \,,
	\end{equation*} 
it follows that
	\begin{equation*}
	\begin{split}
		\norm{y}_\HapOS^2 
		&\simeq
		\int_{0}^{2\pi} \sum\limits_{j,k \in \Z}  \kl{1+ 2^{-2 j (\alpha+1/2)}} \abs{\spr{y(\cdot,\vphi),\psi_{j,k}}_{\Lt(\R)}}^2 \, d\vphi 
		\\
		&\simeq
		\int_{0}^{2\pi} \sum\limits_{j,k \in \Z}  \kl{1+ 2^{-2 j (\alpha+1/2)}}\abs{\sum\limits_{l=1}^\infty \spr{y,\fjkl}_\LtOD w_l(\vphi)}^2 \,. d\vphi 
	\end{split}
	\end{equation*} 	
Hence, since the set $\{w_l\}_{l\in\N}$ forms an orthonormal basis over $\LtT$ we obtain that
	\begin{equation*}
		\norm{y}_\HapOS^2 
		\simeq
		\sum\limits_{j,k \in \Z} \sum\limits_{l=1}^\infty \kl{1+ 2^{-2 j (\alpha+1/2)}}  \abs{ \spr{y,\fjkl}_\LtOD }^2 \,.
	\end{equation*} 
Consequently, for the special case $\beta = 0$ and with the choice
	\begin{equation*}
	\begin{split}
		\Kl{e_k}_{k\in\N} &:= \{\ejkl\}_{j,k\in \Z\,, l \in \N} \,,
		\qquad
		\Kl{f_k}_{k\in\N} := \{\fjkl\}_{j,k\in \Z\,, l \in \N} \,,
		\\
		\{\ak\}_{k\in\N} &:= \{ 1+ 2^{-2 j (\alpha+1/2)}  \}_{j,k\in \Z\,, l \in \N} \,,
	\end{split}
	\end{equation*}
all assumptions of Theorem~\ref{thm_main_Radon} are satisfied, which thus yields the assertions.
\end{proof}

\vspace{0pt}
\begin{remark}
The explicit representation of $\AD y$ given in \eqref{dec_ADy_Radon_specific} can be used as the basis of an efficient numerical routine for solving the tomography problem $A x = y$. For example, one can replace the infinite sums over the indices $j,k,l$ by finite sums, and pre-compute a numerical approximation of each dual frame function $\ejklt$ via \eqref{eq_dual_N_approx}. Then, for each right-hand side $y$, one only needs to compute the coefficients $\spr{y,\fjkl}_\LtOS$ and sum up according to \eqref{dec_ADy_Radon_specific}. If one e.g.\ chooses an exponential basis for $\Kl{w_l}_{l\in\N}$, then these coefficients can be efficiently computed using the (fast) Fourier and wavelet transforms. Hence, in this case an efficient implementation of \eqref{dec_ADy_Radon_specific} for computing $\AD y$ is possible.
\end{remark} 
\vspace{7pt}

\begin{remark}
Note that for $x \in \HazOD$ there holds $\supp{Ax} \subseteq [-1,1]\times[0,2\pi)$. Hence, redefining the tomography operator $A : \HazOD \to \Lt([-1,1]\times[0,2\pi))$ one can define 
	\begin{equation*}\label{def_Radon_wjk}
		w_{j,k}(s,\vphi) := \frac{1}{2\sqrt{ \pi}}\exp(i j \pi s) \exp(ik \vphi) \,,
		\qquad
		v_{j,k} := \kl{1+\abs{j}^2}^{\alpha/2} A^* w_{j,k} \,.
	\end{equation*} 
and use the fact (see e.g.~\cite{Natterer_2001}) that
	\begin{equation*}
		\norm{z}_{\HazOo}^2
		\simeq \pi \sum_{j\in\Z} \kl{1+\pi^2\abs{j}^2}^{\alpha} \abs{\hat{z}(\pi j)}^2
		\simeq
		\sum_{j\in\Z} \kl{1+\abs{j}^2}^{\alpha} \abs{\spr{z,e^{i \pi j \cdot} }_\LtOo}^2 \,,
	\end{equation*}
in order to obtain the alternative decomposition
	\begin{equation*}\label{dec_ADy_Radon_wjk}
		\AD y = \sum\limits_{j,k\in\Z} \kl{1+\abs{j}^2}^{\alpha/2} \spr{y,w_{j,k}}_\LtOS \tilde{v}_{j,k} \,,
	\end{equation*}
which can again be efficiently numerically implemented using the fast Fourier transform.
\end{remark}
\vspace{7pt}

Alternatively, using the decomposition approach based on the operator $L$ we obtain

\begin{theorem}\label{thm_main_Radon_L}
Let the Radon transform $A : \HazOD \to \HbOS$ be defined as in \eqref{Radon_A} for some $0 \leq \alpha,\beta \in \R$ satisfying $\beta \leq \alpha + 1/2$. Furthermore, let the functions $f_k \in \HapOS$ be such that the set $\{f_k \}_{k\in\N}$ forms a frame over $\HbOS$.	Furthermore, let $E : \HapOS \to \HbOS$ denote the embedding operator, let $L := (EE^*)^{-1/2}$ and define the functions $e_k := A^* L f_k $. Then the set $\{e_k\}_{k \in \Z}$ forms a frame over $\HazOD$ and there holds
	\begin{equation*}
		A x
		=L^{-1}\kl{\sum\limits_{k \in \Z}  \spr{x,e_k}_\HazOD \fkt }\,.
	\end{equation*}
Furthermore, for any $y \in \HapOS$ the unique solution of \eqref{Ax=y} is given by
	\begin{equation*}
		\ADb y = \sum\limits_{j,k\in\Z}  \spr{L y,f_k}_\HbOS \ekt  \,.
	\end{equation*}
\end{theorem}
\begin{proof}
As in the proof of Theorem~\ref{thm_main_Radon} we have that $A$ is a well-defined bounded linear operator and that \eqref{cond_A_stability} is satisfied with $X = \HazOD$, $Y = \HbOS$, and $Z = \HapOS$. Furthermore, since $\beta \leq \alpha +1$ we have that $\HapOS$ is a dense subspace of $\HbOS$ and thus also \eqref{cond_Y_Z} holds. Hence, Assumption~\ref{ass_main_III} is satisfied, which implies that Theorem~\ref{thm_main_IIIII} is applicable, which yields the assertions.
\end{proof}

\subsection{Application to Atmospheric Tomography}\label{sect_tomo_atmos}

Atmospheric tomography plays an important part in many Adaptive Optics (AO) systems for the improvement of the imaging quality of earthbound astronomical telescopes such as the Extremely Large Telescope (ELT) \cite{ELT_2020} of the European Southern Observatory (ESO), currently under construction in the Atacama desert in Chile. Based on measurements of the incoming light of both Natural Guide Stars (NGS) and artificially created Laser Guide Stars (LGS) in the vicinity of an object of interest, one aims at reconstructing the atmospheric turbulence on a finite number of turbulent layers, in order to then suitably adjust deformable mirrors in such a way that the incoming wavefronts are corrected (flattened) after reflection on these mirrors. Since the atmosphere is constantly changing, this has to be done in real time. For details we refer to \cite{Roddier_1999,Roggemann_Welsh_1996,Ellerbroek_Vogel_2009}.

Mathematically, the atmospheric tomography problem can be written as a linear operator equation of the form \eqref{Ax=y}, using the \emph{atmospheric tomography operator} \cite{Neubauer_Ramlau_2017,Hubmer_Ramlau_2020} 
	\begin{equation*}\label{def_Tomo_A}
	\begin{split}
		A \, : \, \D(A) :=& \prod\limits_{l=1}^L \LtOl \to \LtOA^G \,,
		\\
		\phi &\mapsto \vphi_g = (A\phi)_g(r)
		:=
		\sum\limits_{l=1}^{L}\phi_l( \clg r+\aghl)  \,,
		\qquad g = 1,\dots,G \,,
	\end{split}
	\end{equation*}
where $\phi = (\phi_l)_{l=1,\dots,L}$ denotes the turbulence layers and $\vphi = (\vphi_g)_{g=1,\dots,G}$ are the incoming wavefronts. Here, $L$ denotes the number of atmospheric layers, located at the  heights $h_l$, $G$ denotes the total number of guide stars with corresponding view directions $\ag =(\agx,\agy)\in \R^2$, and $\clg$ are constants depending on the layer and the guide star. Furthermore, the domain $\OA \subset \R^2$ denotes the telescope aperture, and
	\begin{equation*}
	\begin{split}
		\Ol := \bigcup\limits_{g = 1}^G  \OAaghl \,,
		\qquad
		\text{where}
		\qquad
		\OAaghl := \Kl{r \in \R^2 \, : \, \frac{r - \aghl}{\clg}  \in \OA} \,.
	\end{split}
	\end{equation*}
For more details on this setting we refer to \cite{Neubauer_Ramlau_2017,Hubmer_Ramlau_2020}  and the references therein.

It follows from \cite[Theorem~3.1]{Neubauer_Ramlau_2017} that the operator $A$ is not compact with respect to the canonic inner products and hence, a singular system does not necessarily need to exist. Moreover, to our knowledge neither a Wavelet-Vaguelette nor a similar decomposition of this operator is known. However, it was recently shown in \cite{Hubmer_Ramlau_2020} that a frame decomposition of $A$ is possible. The corresponding frames are built from the functions
	\begin{align*}
		w_{jk}(x,y) &:= \frac{1}{2T} \exp(ij\pi x/T) \exp(ik\pi y/T) \,, 
		\\
		w_{jk,lg}(x,y) &:= \clg^{-1} w_{jk}((x,y)/\clg)I_{\clg \OA + \aghl}(x,y) \,.,
	\end{align*}
where $I_{\clg \OA + \aghl}(x,y)$ denotes the indicator function of the domain $\clg \OA + \aghl$. It was shown that if $T \geq 0$ is chosen large enough, then the set $\{w_{jk}\}_{jk\in\Z}$ forms a tight frame with frame bound $1$ over $\LtOA$ and the sets $\{w_{jk,lg}\}_{jk\in\Z,g=1,\dots,G}$ form frames with frame bounds $C_1 = 1$ and $C_2 = G$ over $\LtOl$. Furthermore, one obtains
	\begin{equation*}
		\spr{(A \phi)_g,w_{jk}}_\LtOA = (2T) \sum\limits_{l=1}^L \clg^{-1} w_{jk}(\aghl/\clg) \spr{\phi_l,w_{jk,lg}}_{\LtOl} \,,
	\end{equation*}
and thus the generalization \eqref{eq_frames_generalized} of condition \eqref{cond_frames} holds. Hence, Theorem~\ref{thm_main_II} is applicable, which yields the same results as the ones presented in \cite[Thm.~4.8]{Hubmer_Ramlau_2020}.

Similarly, the authors of \cite{Neubauer_Ramlau_2017} derived a singular-value-type decomposition of what they called the \emph{periodic} atmospheric tomography operator, which is defined by
	\begin{equation}\label{def_Tomo_At}
	\begin{split}
		\At \, : \, \LtOT^L &\to \LtOT^G \,,
		\\
		\phi &\mapsto \vphi_g = (\At\phi)_g(r)
		:=
		\sum\limits_{l=1}^{L}\phi_l( r+\aghl)  \,,
		\qquad g = 1,\dots,G \,,
	\end{split}
	\end{equation}
where $\OT := [-T,T]^2$ for some $T \geq 0$ sufficiently large and assuming that functions in $\LtOT$ are periodic. In particular, due to \cite[Proposition~4.1]{Neubauer_Ramlau_2017} there holds
	\begin{equation*}
		\spr{(\At \phi)_g,w_{jk}}_\LtOT 
		=
		\sum\limits_{l=1}^{L}  w_{jk}(\agx h_l,\agy h_l) \spr{\phi_l,w_{jk}}_\LtOl \,.
	\end{equation*} 
Since the sets $\{w_{jk}\}_{jk\in\Z}$ form orthonormal bases and thus tight frames over $\LtOT$ with frame bound $1$, it follows that Assumption~\ref{ass_main} is satisfied. Hence, Corollary~\ref{corollary_main_I} is applicable, and we recover the same decomposition and reconstruction results as in \cite{Neubauer_Ramlau_2017}.

The periodic atmospheric tomography operator as defined in \eqref{def_Tomo_At} only covers settings without LGSs. A similar operator, which can be used to treat settings with only LGSs was also considered in \cite{Hubmer_Ramlau_2020}, and the derived frame decomposition results again fit into the theoretical framework developed in this paper.

Lastly, the authors of \cite{Niebsch_Ramlau_2020} recently proposed an approach for using atmospheric tomography in Single Conjugate Adaptive Optics (SCAO), a specific AO setting using only a single guide star and thus not naturally allowing for atmospheric tomography. However, based on a time-series of wavefront measurements together with an estimate of the windspeed on each atmospheric layer, they developed a method for incorporating atmospheric tomography which lead to an improvement in imaging quality. Their approach uses the same operator $\At$ as in \eqref{def_Tomo_At}, but with the parameters $\aghl$ replaced by the corresponding windshift vectors. Since this does not entail any essential structural changes of the problem, the results of \cite{Neubauer_Ramlau_2017} and our theoretical results on the frame decomposition are applicable also in that case.

\section{Conclusion}\label{sect_conclusion}

In this paper, we considered the decomposition of bounded linear operators on Hilbert spaces in terms of functions forming frames. The resulting frame decomposition encodes information on the structure and ill-posedness of the problem and can be used as the basis for the design and implementation of efficient numerical solution methods. In contrast to the singular-value decomposition, the presented frame decomposition can be derived explicitly for a wide class of operators, in particular for those satisfying a certain stability condition. In order to show the usefulness of this approach, we considered different examples from computerized and atmospheric tomography.

\section{Support}

S. Hubmer and R. Ramlau were (partly) funded by the Austrian Science Fund (FWF): F6805-N36. The authors would like to thank Dr.\ Stefan Kindermann for valuable discussions on some theoretical questions which arose during the writing of this manuscript. 

\bibliographystyle{plain}
{\footnotesize
\bibliography{mybib}

\begin{thebibliography}{10}

\bibitem{Abramovich_Silverman_1998}
F.~Abramovich and B.~W. Silverman.
\newblock {Wavelet Decomposition Approaches to Statistical Inverse Problems}.
\newblock {\em Biometrika}, 85(1):115--129, 1998.

\bibitem{Alt_2016}
H.~W. Alt and R.~N{\"u}rnberg.
\newblock {\em {Linear Functional Analysis: An Application-Oriented
  Introduction}}.
\newblock Universitext. Springer London, 2016.

\bibitem{Candes_Donoho_2002}
E.~J. Candes and D.~L. Donoho.
\newblock Recovering edges in ill-posed inverse problems: optimality of
  curvelet frames.
\newblock {\em Ann. Statist.}, 30(3):784--842, 2002.

\bibitem{Chaux_Combettes_Pesquet_Wajs_2007}
C.~Chaux, P.~L Combettes, J.-C. Pesquet, and V.~R Wajs.
\newblock A variational formulation for frame-based inverse problems.
\newblock {\em Inverse Problems}, 23(4):1495--1518, 2007.

\bibitem{Christensen_2016}
O.~Christensen.
\newblock {\em {An Introduction to Frames and Riesz Bases}}.
\newblock Applied and Numerical Harmonic Analysis. Springer International
  Publishing, 2016.

\bibitem{Conway_1994}
J.~B. Conway.
\newblock {\em A {C}ourse in {F}unctional {A}nalysis}.
\newblock Graduate Texts in Mathematics. Springer New York, 1994.

\bibitem{Dahmen_1997}
W.~Dahmen.
\newblock Wavelet and multiscale methods for operator equations.
\newblock {\em Acta Numerica}, 6:55--228, 1997.

\bibitem{Daubechies_1992}
I.~Daubechies.
\newblock {\em Ten Lectures on Wavelets}.
\newblock Society for Industrial and Applied Mathematics, Philadelphia, PA,
  1992.

\bibitem{Donoho_1995}
D.~L. Donoho.
\newblock {Nonlinear Solution of Linear Inverse Problems by
  Wavelet–Vaguelette Decomposition}.
\newblock {\em Applied and Computational Harmonic Analysis}, 2(2):101--126,
  1995.

\bibitem{Ellerbroek_Vogel_2009}
B.~L. Ellerbroeck and C.~R. Vogel.
\newblock Inverse problems in astronomical optics.
\newblock {\em Inverse Problems}, 25:063001 (37pp), 2009.

\bibitem{Engl_1997}
H.~W. {Engl}.
\newblock {\em {Integralgleichungen.}}
\newblock Wien: Springer, 1997.

\bibitem{Engl_Hanke_Neubauer_1996}
H.~W. {Engl}, M.~{Hanke}, and A.~{Neubauer}.
\newblock {\em {Regularization of inverse problems.}}
\newblock Dordrecht: Kluwer Academic Publishers, 1996.

\bibitem{ELT_2020}
{European Southern Observatory (ESO)}.
\newblock {ESO's Extremely Large Telescope}.
\newblock \url{https://www.eso.org/public/teles-instr/elt/}.
\newblock Accessed: 2020-05-13.

\bibitem{Heuser_1981}
H.~G. Heuser.
\newblock {\em {Functional Analysis}}.
\newblock Wiley, 1981.

\bibitem{Hubmer_Ramlau_2020}
S.~Hubmer and R.~Ramlau.
\newblock A frame decomposition of the atmospheric tomography operator.
\newblock {\em Inverse Problems}, 36(9):094001, 2020.

\bibitem{Klann_Ramlau_2008}
E.~Klann and R.~Ramlau.
\newblock Regularization by fractional filter methods and data smoothing.
\newblock {\em Inverse Problems}, 24(2), 2008.

\bibitem{Kudryavtsev_Shestakov_2019}
A.~A. Kudryavtsev and O.~V. Shestakov.
\newblock {Estimation of the Loss Function When Using Wavelet-Vaguelette
  Decomposition for Solving Ill-Posed Problems}.
\newblock {\em Journal of Mathematical Sciences}, 237:804--809, 2019.

\bibitem{Louis_1989}
A.~K. Louis.
\newblock {\em {I}nverse und schlecht gestellte {P}robleme}.
\newblock Teubner Studienb{\"u}cher Mathematik. Vieweg+Teubner Verlag, 1989.

\bibitem{Natterer_2001}
F.~{Natterer}.
\newblock {\em {The Mathematics of Computerized Tomography}}.
\newblock Society for Industrial and Applied Mathematics, Philadelphia, PA,
  2001.

\bibitem{Neira_Rubio_Fern_Plastino_1997}
L.~Neira, J.~Fernandez-Rubio, J.~Fern, and A.~Plastino.
\newblock {Frames: a Maximum Entropy Statistical Estimate of the Inverse
  Problem}.
\newblock {\em Journal of Mathematical Physics}, 38, 1997.

\bibitem{Neubauer_Ramlau_2017}
A.~Neubauer and R.~Ramlau.
\newblock A singular-value-type decomposition for the atmospheric tomography
  operator.
\newblock {\em SIAM Journal on Applied Mathematics}, 77(3):838--853, 2017.

\bibitem{Niebsch_Ramlau_2020}
J.~Niebsch and R.~Ramlau.
\newblock {Tomographic Reconstruction for Single Conjugate Adaptive Optics}.
\newblock In T.~Schuster B.~Kaltenbacher and A.~Wald, editors, {\em
  {Time-dependent problems in Imaging and Parameter Identification}}. Springer,
  Heidelberg, 2020.
\newblock To appear.

\bibitem{Ram_Cohen_Elad_2014}
I.~Ram, I.~Cohen, and M.~Elad.
\newblock {Patch-Ordering-Based Wavelet Frame and Its Use in Inverse Problems}.
\newblock {\em IEEE Transactions on Image Processing}, 23(7):2779--2792, 2014.

\bibitem{Ramlau_Koutschan_Hofmann_2020}
R.~Ramlau, C.~Koutschan, and B.~Hofmann.
\newblock {On the Singular Value Decomposition of n-Fold Integration
  Operators}.
\newblock In J.~Cheng, S.~Lu, , and M.~Yamamoto, editors, {\em Inverse Problems
  and Related Topics}, pages 237--256. Springer Singapore, 2020.

\bibitem{Ramlau_Teschke_2004_1}
R.~Ramlau and G.~Teschke.
\newblock {Regularization of Sobolev Embedding Operators and Applications to
  Medical Imaging and Meteorological Data. Part I: Regularization of Sobolev
  Embedding Operators}.
\newblock {\em Sampling Theory in Signal and Image Processing}, 3(2):175--195,
  2004.

\bibitem{Ramlau_Teschke_2004_2}
R.~Ramlau and G.~Teschke.
\newblock {Regularization of Sobolev Embedding Operators and Applications to
  Medical Imaging and Meteorological Data. Part II: Regularization
  Incorporating Noise with Applications in Medical Imaging and Meteorological
  Data}.
\newblock {\em Sampling Theory in Signal and Image Processing}, 3(3):205--226,
  2004.

\bibitem{Roddier_1999}
F.~Roddier.
\newblock {\em Adaptive optics in astronomy}.
\newblock Cambridge University Press, 1999.

\bibitem{Roggemann_Welsh_1996}
M.~C. Roggemann and B.~Welsh.
\newblock {\em Imaging through turbulence}.
\newblock CRC Press laser and optical science and technology series, CRC Press,
  1996.

\bibitem{Stevenson_2003}
R.~Stevenson.
\newblock {Adaptive Solution Of Operator Equations Using Wavelet Frames}.
\newblock {\em SIAM Journal on Numerical Analysis}, 41, 2003.

\bibitem{Teschke_2005}
G.~Teschke.
\newblock {Multi-Frames in Thresholding Iterations for Nonlinear Operator
  Equations with Mixed Sparsity Constraints}.
\newblock {\em DFG-SPP-1114 preprint 131}, 2005.

\bibitem{Teschke_2007}
G.~Teschke.
\newblock Multi-frame representations in linear inverse problems with mixed
  multi-constraints.
\newblock {\em Applied and Computational Harmonic Analysis}, 22:43--60, 2007.

\bibitem{Zhariy_2009}
M.~Zhariy.
\newblock {\em {Adaptive Frame Based Regularization Methods For Linear
  Ill-Posed Inverse Problems}}.
\newblock PhD thesis, Universit\"at Bremen, 2009.

\end{thebibliography}
}

\end{document}